\documentclass[11pt]{preprint_DP}
\usepackage[full]{textcomp}
\usepackage[osf]{newtxtext}
\usepackage{comment}

\usepackage{amssymb}
\usepackage{mathtools}
\usepackage{hyperref}
\usepackage{breakurl}
\usepackage{mhenvs}
\usepackage{mhequ}
\usepackage{mhsymb}
\usepackage{booktabs}
\usepackage{tikz}
\usepackage{mathrsfs}
\usepackage{longtable}

\usepackage{microtype}
\usepackage{comment}
\usepackage{wasysym}
\usepackage{centernot}
\usepackage{enumitem}
\usepackage{bm}
\usepackage{stackrel}
\usepackage{halloweenmath}

\newcommand{\eqlaw}{\stackrel{\mbox{\tiny law}}{=}}

\DeclareMathAlphabet{\mathbbm}{U}{bbm}{m}{n}

\DeclareFontFamily{U}{BOONDOX-calo}{\skewchar\font=45 }
\DeclareFontShape{U}{BOONDOX-calo}{m}{n}{
  <-> s*[1.05] BOONDOX-r-calo}{}
\DeclareFontShape{U}{BOONDOX-calo}{b}{n}{
  <-> s*[1.05] BOONDOX-b-calo}{}
\DeclareMathAlphabet{\mcb}{U}{BOONDOX-calo}{m}{n}
\SetMathAlphabet{\mcb}{bold}{U}{BOONDOX-calo}{b}{n}
\setlist{noitemsep,topsep=4pt}

\newcommand{\mrd}{\,\mathrm{d}}

\makeatletter
\def\DeclareSymbol#1#2#3{\expandafter\gdef\csname MH@symb@#1\endcsname{\tikz[baseline=#2,scale=0.15,draw=symbols,line join=round]{#3}}\expandafter\gdef\csname MH@symb@#1s\endcsname{\scalebox{0.7}{\tikz[baseline=#2,scale=0.15,draw=symbols,line join=round]{#3}}}}
\def\<#1>{\csname MH@symb@#1\endcsname}
\makeatother

\DeclareSymbol{0}{-2.4}{\node[dot] {};}
\DeclareSymbol{1}{0}{\draw[white] (-.5,0) -- (.5,0); \draw (0,0)  -- (0,1.5) node[sdot] {};}
\DeclareSymbol{11}{0}{\draw (-0.5,1.2) node[sdot] {} -- (0,0) -- (0.5,1.2) node[sdot] {};}
\DeclareSymbol{111}{0}{\draw (0,0) -- (0,1.2) node[sdot] {}; \draw (-.7,1) node[sdot] {} -- (0,0) -- (.7,1) node[sdot] {};}
\DeclareSymbol{131}{-3}{\draw (0,0) -- (0,-1) -- (1,0) node[sdot] {}; \draw (0,0) -- (0,-1) -- (-1,0) node[sdot] {}; \draw (0,0) -- (0,1.2) node[sdot] {}; \draw (-.7,1) node[sdot] {} -- (0,0) -- (.7,1) node[sdot] {};}
\DeclareSymbol{11}{0}{\draw (-0.5,1.2) node[sdot] {} -- (0,0) -- (0.5,1.2) node[sdot] {};}
\DeclareSymbol{12}{-3}{\draw (-.8,1) node[sdot] {} -- (0,0) -- (0,-1); \draw (1,0) node[sdot] {} -- (0,-1) -- (-1,0) node[sdot] {};}
\DeclareSymbol{30}{-3}{\draw (0,0) -- (0,-1); \draw (0,0) -- (0,1.2) node[sdot] {}; \draw (-.7,1) node[sdot] {} -- (0,0) -- (.7,1) node[sdot] {};}
\DeclareSymbol{10}{-3}{\draw (-.8,1) node[sdot] {} -- (0,0) -- (0,-1);}
\DeclareSymbol{22}{-3}{\draw (0,0.3) -- (0,-1) -- (1,0) node[sdot] {}; \draw (0,0.3) -- (0,-1) -- (-1,0) node[sdot] {};\draw (-.7,1) node[sdot] {} -- (0,0.3) -- (.7,1) node[sdot] {};}
\DeclareSymbol{211}{0}{\draw (-0.5,2.4) node[sdot] {} -- (-1,1.6);\draw (-1.5,2.4) node[sdot] {} -- (-0.5,0.8); \draw (0,1.6) node[sdot] {} -- (-0.5,0.8) -- (0,0) -- (0.5,0.8) node[sdot] {};}

\DeclareSymbol{Xi22}{0.5}{\draw (0,0) node[xi] {} -- (-1,1) node[not] {} -- (0,2) node[xi] {};}

\DeclareSymbol{Xi2}{-2}{\draw (0,-0.25) node[xi] {} -- (-1,1) node[xi] {};}
\DeclareSymbol{Xi3}{0}{\draw (0,0) node[xi] {} -- (-1,1) node[xi] {} -- (0,2) node[xi] {};}
\DeclareSymbol{Xi4}{2}{\draw (0,0) node[not] {} -- (-1,1) node[xi] {} -- (0,2) node[xi] {} -- (-1,3) node[xi] {};}
\DeclareSymbol{Xi2X}{-2}{\draw (0,-0.25) node[xi] {} -- (-1,1) node[xix] {};}
\DeclareSymbol{XXi2}{-2}{\draw (0,-0.25) node[xix] {} -- (-1,1) node[xi] {};}

\DeclareSymbol{IXi^2}{-1}{\draw (-1,1) node[xi] {} -- (0,0);
\draw[kernels2] (1,1) node[xi] {} -- (0,0) node[not] {};}

\DeclareSymbol{XiX}{-2.8}{\node[xibx] {};}
\DeclareSymbol{tauX}{-2.8}{ \draw[kernels2] (0,0) node[xibx] {};}
\DeclareSymbol{Xi}{-2.8}{\node[xib] {};}
\DeclareSymbol{XiY}{-2.8}{\node[xie] {};}
\DeclareSymbol{XiZ}{-2.8}{\node[xid] {};}
\DeclareSymbol{IXiX}{0}{\draw (0,-0.25) node[not] {} -- (0,1.5) node[xix] {};}

\newcommand{\cut}{\mathfrak{C}}

\newcommand{\mcE}{\mathcal{E}}

\newcommand{\mcC}{\mathcal{C}}

\newcommand{\mcS}{\mathcal{S}}

\newcommand{\mcL}{\mathcal{L}}

\newcommand{\mcD}{\mathcal{D}}

\newcommand{\T}{\mathbf{T}}

\def\${|\!|\!|}

\def\scal#1{{\langle#1\rangle}}

\def\CS{\mathcal{S}}

\newcommand{\mfB}{\mathfrak{B}}

\newcommand{\mfA}{\mathfrak{A}}
\newcommand{\mfC}{\mathfrak{C}}
\newcommand{\mfR}{\mathfrak{R}}

\newcommand{\mfp}{\mathfrak{p}}

\def\cC{\mathscr{C}}

\def\combplus[#1,#2,#3,#4]{\binom{#1\ {\scriptstyle #4} }{#2\ #3}}

\def\singlescalegenvert[#1,#2]{\hat{H}^{#2}_{#1}}
\def\multiscalegenvert[#1,#2]{H^{#2}_{#1}}

\def\nr[#1]{\tilde{N}[#1]} %nr stands for non-root
\def\inn[#1]{\mathring{N}[#1]}
\def\nrinn[#1]{\hat{N}_{#1}} %nrinn stands for non-root inner nodes
\def\nrmod[#1,#2]{\tilde{N}_{#1}(#2)}
\def\nrinnmod[#1,#2]{\hat{N}_{#1}(#2)}

\def\ident[#1]{\underline{#1}}

\def\mylink#1#2{\mathrel{\vbox{\offinterlineskip\ialign{%
    \hfil##\hfil\cr
    $\scriptscriptstyle#1$\cr
    \noalign{\kern0.1ex}
    $#2$\cr
}}}}
\def\mysublink[#1]#2#3{\mathrel{\vbox{\offinterlineskip\ialign{%
    \hfil##\hfil\cr
    $\scriptscriptstyle#2$\cr
    \noalign{\kern0.1ex}
    $#3$\cr
    \noalign{\kern-0.2ex}
    \smash{\raisebox{-\height}{\hbox{$\scriptscriptstyle #1$}}}\cr
    \noalign{\kern0.2ex}
}}}}

\def\fon[#1]{\cC_{#1}}

\def\mincompproj[#1]{\mfp_{#1}}

\def\Proj_#1{\mathop{\mathrm{Proj}_{#1}}}

\def\negrenorm[#1]{\mfR_{#1}}
\def\topnegrenorm[#1]{\overline{\mfR}_{#1}}

\def\quotedge[#1]{E^{q}_{#1}}

\def\posrenorm[#1]{\mcC_{#1}}
\def\topposrenorm[#1]{\overline{\mcC_{#1}}}
\def\cutsmod[#1]{\mathbb{C}_{+,#1}}

\def\fullcutsmod[#1]{\cut_{#1}}

\colorlet{symbols}{black!50}
\colorlet{testcolor}{green!60!black}
\colorlet{darkblue}{blue!60!black}
\colorlet{darkgreen}{green!60!black}

\colorlet{redkernel}{red!80}

\def\symbol#1{\textcolor{symbols}{#1}}
\def\1{\mathbf{\symbol{1}}}

\usetikzlibrary{shapes.misc}
\usetikzlibrary{shapes.symbols}
\usetikzlibrary{shapes.geometric}
\usetikzlibrary{decorations}
\usetikzlibrary{decorations.markings}

\usetikzlibrary{calc}
\def\mathcolor#1#{\mathcoloraux{#1}}
\newcommand*{\mathcoloraux}[3]{%
  \protect\leavevmode
  \begingroup
    \color#1{#2}#3%
  \endgroup
}

\makeatletter
\pgfdeclareshape{crosscircle}
{
  \inheritsavedanchors[from=circle] % this is nearly a circle
  \inheritanchorborder[from=circle]
  \inheritanchor[from=circle]{north}
  \inheritanchor[from=circle]{north west}
  \inheritanchor[from=circle]{north east}
  \inheritanchor[from=circle]{center}
  \inheritanchor[from=circle]{west}
  \inheritanchor[from=circle]{east}
  \inheritanchor[from=circle]{mid}
  \inheritanchor[from=circle]{mid west}
  \inheritanchor[from=circle]{mid east}
  \inheritanchor[from=circle]{base}
  \inheritanchor[from=circle]{base west}
  \inheritanchor[from=circle]{base east}
  \inheritanchor[from=circle]{south}
  \inheritanchor[from=circle]{south west}
  \inheritanchor[from=circle]{south east}
  \inheritbackgroundpath[from=circle]
  \foregroundpath{
    \centerpoint%
    \pgf@xc=\pgf@x%
    \pgf@yc=\pgf@y%
    \pgfutil@tempdima=\radius%
    \pgfmathsetlength{\pgf@xb}{\pgfkeysvalueof{/pgf/outer xsep}}%  
    \pgfmathsetlength{\pgf@yb}{\pgfkeysvalueof{/pgf/outer ysep}}%  
    \ifdim\pgf@xb<\pgf@yb%
      \advance\pgfutil@tempdima by-\pgf@yb%
    \else%
      \advance\pgfutil@tempdima by-\pgf@xb%
    \fi%
    \pgfpathmoveto{\pgfpointadd{\pgfqpoint{\pgf@xc}{\pgf@yc}}{\pgfqpoint{-0.707107\pgfutil@tempdima}{-0.707107\pgfutil@tempdima}}}
    \pgfpathlineto{\pgfpointadd{\pgfqpoint{\pgf@xc}{\pgf@yc}}{\pgfqpoint{0.707107\pgfutil@tempdima}{0.707107\pgfutil@tempdima}}}
    \pgfpathmoveto{\pgfpointadd{\pgfqpoint{\pgf@xc}{\pgf@yc}}{\pgfqpoint{-0.707107\pgfutil@tempdima}{0.707107\pgfutil@tempdima}}}
    \pgfpathlineto{\pgfpointadd{\pgfqpoint{\pgf@xc}{\pgf@yc}}{\pgfqpoint{0.707107\pgfutil@tempdima}{-0.707107\pgfutil@tempdima}}}
  }
}
\makeatother

\definecolor{connection}{rgb}{0.7,0.1,0.1}

\tikzset{
root/.style={circle,fill=black!50,inner sep=0pt, minimum size=3mm},
        dot/.style={circle,fill=black,inner sep=0pt, minimum size=1.2mm},
        sdot/.style={circle,fill=black,inner sep=0pt,minimum size=.5mm},
        dotred/.style={circle,fill=black!50,inner sep=0pt, minimum size=2mm},
        var/.style={circle,fill=black!10,draw=black,inner sep=0pt, minimum size=3mm},
        kernel/.style={semithick,shorten >=2pt,shorten <=2pt},
        kernel1/.style={thick},
        kernels/.style={snake=zigzag,shorten >=2pt,shorten <=2pt,segment amplitude=1pt,segment length=4pt,line before snake=2pt,line after snake=5pt,},
        rho/.style={densely dashed,semithick,shorten >=2pt,shorten <=2pt},
           testfcn/.style={dotted,semithick,shorten >=2pt,shorten <=2pt},
           tau/.style={circle,inner sep=1pt,draw=black,fill=white,text=black,thin},
        renorm/.style={shape=circle,fill=white,inner sep=1pt},
        labl/.style={shape=rectangle,fill=white,inner sep=1pt},
        xic/.style={very thin,circle,fill=symbols,draw=black,inner sep=0pt,minimum size=1.2mm},
        xi/.style={very thin,circle,fill=blue!10,draw=black,inner sep=0pt,minimum size=1.2mm},
        xix/.style={crosscircle,fill=blue!10,draw=black,inner sep=0pt,minimum size=1.2mm},
	xib/.style={very thin,circle,fill=blue!10,draw=black,inner sep=0pt,minimum size=1.6mm},
	xie/.style={very thin,circle,fill=green!50!black,draw=black,inner sep=0pt,minimum size=1.6mm},
	xid/.style={very thin,circle,fill=symbols,draw=black,inner sep=0pt,minimum size=1.6mm},
	xibx/.style={crosscircle,fill=blue!10,draw=black,inner sep=0pt,minimum size=1.6mm},
	kernels2/.style={very thick,draw=connection,segment length=12pt},
	not/.style={thin,circle,fill=symbols,draw=connection,fill=connection,inner sep=0pt,minimum size=0.5mm},
	>=stealth,
  }

\newtheorem{assumption}[lemma]{Assumption}

\colorlet{darkblue}{blue!90!black}
\colorlet{darkred}{red!90!black}
\colorlet{darkgreen}{green!70!black}

\def\K{\mathfrak{K}}

\def\${|\!|\!|}
\def\Wick#1{\,\colon\!\! #1 \colon\!}

\newcommand{\Law}{\mathrm{Law}}

\def\?{{\color{red}?}}

\newcommand{\floor}[1]{\lfloor #1 \rfloor}

\usepackage{stmaryrd}
\def\fancynorm#1{{\talloblong #1 \talloblong}}

\newcommand{\bPsi}{\boldsymbol{\Psi}}

\newcommand{\Cov}{\mathrm{Cov}}

\def\slash{\kern0.18em/\penalty\exhyphenpenalty\kern0.18em}
\def\dash{\kern0.18em--\penalty\exhyphenpenalty\kern0.18em}

\newtheorem{example}[lemma]{Example}

\let\basepoint\logof
\def\logof{\mathord{{\basepoint}}} % Turns it into the behaviour of an ordinary math symbol

\title{Non-Gaussianity of invariant measures to SPDEs in Da Prato--Debussche regime}
\author{Ajay~Chandra$^1$ and Ilya~Chevyrev$^{2}$}

\institute{Imperial College London, UK \and SISSA, Trieste, Italy}

\date{\today}
\begin{document}
\maketitle
\begin{abstract}
We propose an elementary method to show non-Gaussianity of invariant measures of parabolic stochastic partial differential equations with polynomial non-linearities in the Da Prato--Debussche regime. The approach is essentially algebraic and involves using the generator equation of the SPDE at stationarity. Our results in particular cover the $\Phi^4_\delta$ measures in dimensions $\delta<\frac{14}{5}$, which includes cases where the invariant measure is singular with respect to the invariant measure of the linear solution.
\end{abstract}
\setcounter{tocdepth}{2}

\tableofcontents

\section{Introduction}\label{sec: intro}
\textbf{Background.}
There has been significant interest in recent years in studying Euclidean quantum field theories (EQFT) using methods from stochastic analysis.
One of the most prominent of these methods is parabolic stochastic quantisation, wherein one studies an EQFT using its Langevin dynamic (or stochastic quantisation equation),
an idea first proposed by Parisi--Wu \cite{ParisiWu}.

By EQFT, we mean a Gibbs-type probability measure on $\CS'(\R^d)$ given heuristically by
\begin{equ}[eq:Gibbs]
\mu(\mrd u) \propto \exp(-S(u)) \mcD u
\end{equ}
where $\mcD u$ is a formal Lebesgue measure on $\CS'(\R^d)$ and $S\colon \CS'(\R^d)\to \R$ is an action.
% often taking the form $S(u) = \int_{\R^d} F(u,\nabla u,\ldots)$ where $F$ is a function on $u$ and its derivatives.
The Langevin dynamic, in this case, is the stochastic (partial) differential equation
\begin{equ}[eq:Langevin]
\partial_t u = -\frac12 \nabla S(u) + \xi
\end{equ}
where $\xi$ is a space-time white noise on $\R\times\R^d$ with respect to some metric on $\CS'(\R^d)$ (we will always take the $L^2$ metric) and $\nabla S$ is the gradient of $S$ with respect to the same metric.

Of course, both \eqref{eq:Gibbs} and \eqref{eq:Langevin} are formal as written and giving rigorous meaning to them is often challenging.
For example, the $\Phi^{p+1}_d$ measures, for which $S(u) = \int_{\R^d}\frac{1}{2}|\nabla u|^2 + \frac{1}{p+1}|u|^{p+1}$ with $p\geq 3$ an odd integer,
were the subject of intensive investigation in the field of constructive quantum field theory, see e.g. \cite{GlimmJaffe}. 
Their construction in dimension $d\geq 2$ is non-trivial and requires a renormalisation of the Lagrangian density $\frac12|\nabla u|^2 + \frac{1}{p+1}|u|^{p+1}$.
For such EQFT's, the Langevin dynamic \eqref{eq:Langevin} is then given by the SPDE
\begin{equ}[eq:Langevin_Phi4]
\partial_t u = \Delta u - u^p + \xi\;,
\end{equ}
whose solution theory, even locally-in-time, is non-trivial and requires a renormalisation of the equation for $d\geq 2$ due to the singularity of $u^p$.

A celebrated result of Da Prato--Debussche (DPD) \cite{DPD03_SQ}
gave a rigorous meaning to so-called strong solutions to \eqref{eq:Langevin_Phi4} on the 2-dimensional torus $\T^2$
using smooth approximations. Their result can be approximately stated as follows: let $\xi^\eps$ be a smooth approximation of the white noise $\xi$ on $\R\times\T^2$ at scale $\eps>0$.
Then there exist constants $\sigma^2_\eps$ with $\lim_{\eps\downarrow0}\sigma^2_\eps = \infty$ such that the solution $u^\eps\colon [0,\infty)\times\T^2\to\R$ to
the regularised and renormalised SPDE
\begin{equ}[eq:Langevin_Phi4_renorm]
\partial_t u^\eps = \Delta u^\eps - H^{\sigma^2_\eps}_{p}(u^\eps) + \xi^\eps 
\end{equ}
converges to a Markov process $u$ globally in time for $\mu$-a.e. initial condition, where $H^{\sigma^2_\eps}_{p}$ is the $p$-th Hermite polynomial with variance $\sigma^2_\eps$
and $\mu$ is the $\Phi^p_2$ measure on $\T^2$, which is invariant for $u$.
Their method is based on a simple yet very useful decomposition of the solution, which we review in Section \ref{sec:DPD}.
(The restriction to $\mu$-a.e. initial condition in \cite{DPD03_SQ} comes from the use of the invariant measure in the passage from local to global solutions.
Later, \cite{Mourrat_Weber_17_Phi42} showed existence of global solutions for all initial conditions using only PDE arguments.)

Later on, several solution theories appeared
that were able to solve \eqref{eq:Langevin_Phi4} in dimension $d=3$ and even in fractional dimension $d<4$ (known as the subcritical or super-renormalisable regime), see \cite{Hairer14,GIP15}
and \cite{Kupiainen2016,CH16,CC18_Phi43,BHZ19,BCCH17,OW,Duch21,LOTT24} for a highly incomplete sample of further developments.
These solution theories in large part extend rough path theory \cite{Lyons},
designed to solve singular \emph{ODEs},
to higher dimensions and an important step is to incorporate renormalisation like in \eqref{eq:Langevin_Phi4_renorm},
which is typically not present in the ODE setting (though see \cite{KPZ,BCF18_renorm_SDEs,BCFP19}).
The method of DPD can be seen as an important precursor to these theories,
and the threshold at which one can solve renormalised equations like \eqref{eq:Langevin_Phi4}-\eqref{eq:Langevin_Phi4_renorm} without appealing to rough path-type arguments is now sometimes called the `DPD regime'.
We note that a similar method appeared for the stochastic Navier--Stokes equations \cite{DPD02_SNS}
and that the DPD decomposition can also be useful for non-singular equations in obtaining a priori estimates, see e.g. \cite{CHM23}.

An important motivation for analysing the Langevin dynamic \eqref{eq:Langevin} is its potential use in studying the measure \eqref{eq:Gibbs}.
Indeed, the study of the dynamic has led to new constructions of the $\Phi^4_d$ measures for $d<4$ \cite{MW17Phi43,AK20,GH21,MoinatWeber20,CMW23,DGR_24_fractional,EW_24_fractional},
proofs of new tail bounds \cite{Hairer_Steele_22},
and proof of exponential decay of correlations in exponential models \cite{GHN_24_decay}.
The Langevin dynamic has also been studied for Yang--Mills--Higgs models in dimensions 2 and 3 \cite{Shen18,CCHS_2D,CCHS_3D,CS23_invariant, BringmannCaoHiggs},
the sine-Gordon model \cite{HaoSG, Chandra_Hairer_Shen_18_SG,CFW24,BC24SG}, and tensor field theories \cite{CF_24_tensor}.
For other stochastic analytic approaches to EQFTs, see \cite{BG20, GM24}.

\textbf{Our contribution.}
The main purpose of this article is to present an elementary and, as far as we can tell, new method to prove non-Gaussianity of invariant measures of SPDEs
like \eqref{eq:Langevin_Phi4} in the DPD regime (including noises that are more singular than white noise on $\R\times\T^2$).
Our method is essentially algebraic and uses the generator equation of the SPDE at stationarity, i.e. an infinite-dimensional Euler--Lagrange equation.
This equation, as remarked in \cite{BHST87I}, is essentially an integrated, second order integration by parts (Dyson--Schwinger) identity (see Remark \ref{rem:DS_eq}),
but its use in the current context appears new.

We remark that, by a recent result of \cite{Hairer_24_singular}, the DPD regime, to which our method applies, is larger than the regime in which the invariant measure is absolutely continuous with respect to the Gaussian free field.
(For this latter regime, proving non-Gaussianity is even simpler than our proposed method.)
Furthermore, our method applies to polynomial equations with non-local terms such as the $\Phi^3_2$ model, see Remarks \ref{rem:Phi31} and \ref{rem:Phi32} (see also the end of Section \ref{sec:DPD} for comments on the Sine-Gordon model).

Along the way, we review the local well-posedness result of DPD and make precise the form of the generator of $u$.
There has been interest lately in constructing the generator for singular stochastic equations.
For example, \cite{Gubinelli_Perkowski_20_generator} constructs the generator for the stochastic Burgers equation.
In comparison, the problem of constructing the generator
for \eqref{eq:Langevin_Phi4} in the DPD regime is much simpler but still has subtleties which we aim to clarify.
For example, we need to view $u$ as a Markov process with a state space of `rough distributions', which has empty intersection with smooth functions, and this viewpoint is a bit different from other works we are aware of.

\textbf{Related works.}
Non-Gaussianity is an important property of an EQFT as it is a necessary condition for the physical quantum field theory, obtained by going to Minkowski space from Euclidean space via Wick rotation, to
be an interacting theory.
Our result that the invariant measure of $u$ is non-Gaussian is certainly not new.
In the case of two dimensions, non-Gaussianity is an immediate consequence of the asymptotic moment formula obtained in \cite{Dimock74}, and the work \cite{MS77} carries out the construction and proves non-Gaussianity in the harder three dimensional case. 
Closer to our setting is \cite{BFS83_Phi4}  which uses Dyson--Schwinger equations, in combination with skeleton inequalities, to give constructions and proofs of non-Gaussianity of the $\Phi^4_2$ and $\Phi^4_3$ measures for small coupling strengths.
Our arguments, we believe, are more elementary but do not cover $\Phi^4_3$ (but do apply to all coupling strengths).

Our use of the Langevin dynamic to prove non-Gaussianity is also not the first: \cite{GH21,Hairer_Steele_22, Albeverio_Kusuoka_21_rotation} show non-Gaussianity of $\Phi^4_3$ (at all coupling strengths) by using the dynamic, which is \textit{outside} the DPD regime and is thus, in principle, more general.
Our approach, however, is different and somewhat simpler,
and demonstrates yet another use of the Langevin dynamic in the study of EQFTs.

\textbf{Article structure.}
In Section \ref{sec:classical} we demonstrate the idea of our method on a classically well-posed equation that does not require renormalisation,
namely \eqref{eq:Langevin_Phi4} posed on $\R_+\times\T$ and driven by white noise.

In Section \ref{sec:DPD}, we turn to the DPD regime.
In Section \ref{subsec:DPD_trick} we review the DPD argument, presenting a viewpoint that allows us to make precise the generator of the SPDE in Section \ref{sec:Markov}.
In Section \ref{sec:non_gauss} we use the generator to give an elementary proof of non-Gaussianity of the SPDEs under consideration.

Finally, in Appendix \ref{app:abs_cont},
we show that the invariant measure of the SPDEs we consider are, in general, singular with respect to the invariant measure of the linear solution (i.e. the GFF).

\subsection{Notation}
Let $\T^d=\R^d/\Z^d$ denote the torus.
We will implicitly identify $\T^d$ with $[-\frac12,\frac12)^d$ and identify functions on $\T^d$ with periodic functions on $\R^d$.
We write $\scal{\xi,f}$ for the pairing between a distribution $\xi$ and function $f$.
We let $x\wedge y = \min\{x,y\}$.

We denote $X=O(Y)$ and $X\lesssim Y$ to mean that there exists $C>0$ such that $X\leq CY$, where the dependencies of $C$ are clear from the context.

We write $\CC$ and $\CC^0$ for the space of continuous functions and bounded measurable functions respectively.
For $\alpha>0$, we let $\CC^\alpha(\T^d)$ denote the closure of smooth functions $f\colon \T^d \to \R$ under the usual norm
\begin{equ}
\|f\|_{\CC^\alpha(\T^d)}
\eqdef
\max_{|k|<\floor\alpha} \|\partial^k f\|_\infty + \max_{|k|=\floor{\alpha}}\|\partial^{k} f\|_{\CC^{\alpha-\floor{\alpha}}}
\end{equ}
where, for $\eta\in [0,1]$,
\begin{equ}
\|f\|_{\CC^\eta} = \sup_{x\neq y} |x-y|^{-\eta}|f(x)-f(y)|\;,
\end{equ}
and where $\floor\alpha$ is the floor of $\alpha$ and $\partial^k f$ is the usual $k$-th derivative of $f$ for a multi-index $k\in \N^d$, $\N = \{0,1,\ldots\}$,
and where we denote $|k| = \sum_{i=1}^d k_i$.
% E.g. $\CC^{\frac12}(\T^d)$ is the little H\"older space of order $\frac12$ and $\CC^1(\T^d)$ is the space of functions with a continuous derivative.
We similarly let $\CC^k(\R^d)$ denote space of $k$-times continuously differentiable functions on $\R^d$.

For $\K\subset \R\times\T^d$, we make the same definition for $\CC^\alpha(\K)$ except that, for a multi-index $k = (k_0,\ldots, k_d)\subset \N^{1+d}$, we use the parabolic scaling $|k| = 2k_0 + \sum_{i=1}^d k_i$.

For $\alpha<0$, we write $\CC^\alpha(\T^d) \subset \CD'(\T^d)$
for the closure of smooth functions under the
(inhomogenous) H\"older--Besov norm
\begin{equ}
	\|\xi\|_{\CC^\alpha(\T^d)} = 
	\sup_{x\in\T^d}\sup_{\phi \in \CB^r}
	\sup_{\lambda\in (0,1]} \lambda^{-\alpha}|\scal{\xi,\phi^\lambda_x}|\;,
\end{equ}
where $r=-\floor{\alpha}+1$, $\CB^r$ is the set of all $\phi \in \CC^\infty(\R^d)$ with support in the ball $\{|z|<\frac14\}$ and $\|\phi\|_{\CC^r(\R^d)}\leq 1$,
and where $\phi^\lambda_x\in\CC^\infty(\T^d)$ is given by $\phi^\lambda_x (z) = \lambda^{-d}\phi((z-x)\lambda^{-1})$.
For $\K\subset \R\times\T^d$, we define $\CC^\alpha(\K)$ 
analogously except that the first supremum is over $x\in\K$, and, in the second supremum, $\CB^r$ is taken as the set of all $\phi \in \CC^\infty(\R\times \R^d)$ with support in $\{|z|<\frac14\}$ and $\|\phi\|_{\CC^r(\R \times \R^d)}\leq 1$
and where we denote $\phi^\lambda_x (z) = \lambda^{-d-2}\phi((z_0-x_0)\lambda^{-2},(\tilde z-\tilde x)\lambda^{-1})$
and $z= (z_0,\tilde z)\in \R\times\T^d$.
For any $\alpha\in\R$, if no domain is specified, we let $\CC^\alpha$  denote $\CC^{\alpha}(\T^d)$.

For a random variable $X$, we write $X\sim\mu$ to mean that the law of $X$ is $\mu$.

\section{Non-triviality in the classical case}
\label{sec:classical}

To demonstrate the idea, we use the dynamic to show that invariant measures of a class of classically well-posed SPDEs are non-Gaussian.
The basic example to keep in mind is $\Phi^{p+1}_1$ for $p\geq 3$ an odd integer.

We fix space dimension $d=1$ in this section.
We use the shorthand $O = (-1,2)\times\T\subset\R\times\T$.
Consider the classically well-posed stochastic PDE
\begin{equ}\label{eq:SPDE_classical}
(\partial_t -\Delta) u = P(u) + \xi\;, \qquad u_0\in \CC(\T)\;,
\end{equ}
where $P$ is, for now, any function $P\colon \CC(\T)\to\CC(\T)$ for which there exists $\ell>0$ such that, for all $x,y\in\CC(\T)$,
\begin{equ}[eq:P_Lip]
\|P(x)-P(y)\|_{\infty} \leq \|x-y\|_\infty (2+\|x\|_\infty+\|y\|_\infty)^\ell\;.
\end{equ}
% is a polynomial of the form $P(u)=-u^p + \tilde P(u)$ where $\deg(\tilde P)<p$ and $p$ is an odd integer.
Furthermore, $\xi$ is a stationary space-time Gaussian distribution with covariance $\delta_0 \otimes \Cov_{\xi} \in \CS'(\R\times\T)$
for some $\Cov_{\xi} \in \CS'(\T)$ such that convolution with $\Cov_{\xi}$ can be extended\footnote{Essentially, $\Cov_{\xi}$ should be no worse than a Dirac delta.}to a bounded operator $L^2(\T) \rightarrow L^2(\T)$.
Note that $\xi$ is white in time and $\Cov_\xi$ refers only to the \emph{spatial} covariance of $\xi$.
Moreover,
by a Kolmogorov-type argument (see e.g. \cite[Thm.~2.7]{Chandra_Weber_17_SPDEs}), for all $\alpha<\frac12$ and $p\in [1,\infty)$,
\begin{equ}[eq:xi_moments]
\E\|\xi\|_{\CC^{\alpha-2}(O)}^p < \infty\;.
\end{equ}
The local Lipschtiz bound \eqref{eq:P_Lip} readily implies
the following deterministic local well-posedness result.

\begin{proposition}\label{prop:well-posed}
Suppose $\alpha \in (0,\frac12)$ and $\xi\in\CC^{\alpha-2}(O)$ and $u_0\in \CC(\T)$.
Then there exists $\kappa>0$, depending only on $\alpha$, such that, for $T = (2+\|u_0\|_{\infty}+ \|\xi\|_{\CC^{\alpha-2}(O)})^{-1/\kappa} \in (0,1)$,
\eqref{eq:SPDE_classical} admits a unique solution $u \in \CC([0,T],\CC(\T))$ and
\begin{equ}
\sup_{s\in [0,T]}\|u_s\|_{\infty} \leq 2\|u_0\|_{\infty}+1\;.
\end{equ}
\end{proposition}
We further make use of the following Assumption \ref{ass:P} on \emph{global} solutions to \eqref{eq:SPDE_classical}. 
\begin{assumption}\label{ass:P}
For every $u_0\in\CC(\T)$ and $\xi \in \CC^{\alpha-2}(O)$ with $\alpha\in (0,\frac12)$, \eqref{eq:SPDE_classical} admits a unique solution in
$\CC([0,1],\CC(\T))$
and there exists $m>0$ such that, for all $t\in (0,1]$,
\begin{equ}[eq:u_apriori]
\|u_t\|_\infty \leq (2+ \|u_0\|_\infty + \|\xi\|_{\CC^{\alpha-2}(O)} + t^{-1})^m\;.
\end{equ}
\end{assumption}
\begin{remark}\label{rem:assump_P}
If $P$ is a polynomial of degree $p\geq 2$, then Assumption \ref{ass:P} is equivalent to $p$ being odd and that the leading coefficient of $P$ is negative, i.e. $P(x)= \theta x^p + \sum_{i=0}^{p-1} a_ix^i$ with $\theta<0$ and $a_i\in\R$.
In this case, one moreover has the stronger bound
\begin{equ}
\|u_t\|_\infty \leq (2+ \|\xi\|_{\CC^{\alpha-2}(O)} + t^{-1})^m\;,
\end{equ}
see e.g. \cite{MoinatWeber20_RD,MoinatWeber20} where this bound is shown in similar situations.
\end{remark}
By Proposition \ref{prop:well-posed} and \eqref{eq:xi_moments},
\eqref{eq:SPDE_classical} admits local-in-time solutions almost surely,
which, under Assumption \ref{ass:P}, are global.
\begin{lemma}[Generator]\label{lem:generator_classical}
Suppose Assumption \ref{ass:P} holds.
Let $u_0 \in \CC(\T)$ and $u$ solve \eqref{eq:SPDE_classical}.
Consider $\phi_1,\ldots, \phi_k \in \CC^5(\T)$ and $F \in \CC^3(\R^k)$ such that the third derivative of $F$ has at most polynomial growth.
Denote
\begin{equ}
X_t \eqdef (\scal{ u_t,\phi_1}, \ldots, \scal{u_t,\phi_k}) \in \R^k\;.
\end{equ}
Then
\begin{equ}[eq:gen_classical]
\lim_{t\downarrow0}t^{-1}\E[F(X_t) - F(X_0)] = \mcL F(u_0)
\end{equ}
where
\begin{equs}
\mcL F(u_0)
&= \sum_{i=1}^k \partial_i F(X_0)(\scal{ u_0, \Delta\phi_i } + \scal{P(u_0),\phi_i})
\\
&\quad + \frac{1}{2}\sum_{i,j=1}^k \partial_{i,j}^2 F(X_0) \scal{\Cov_{\xi}*\phi_i,\phi_j}\;.
\end{equs}
Furthermore, there exists $q>0$, depending only on $F$ and $P$, such that, for all $t\in (0,1)$,
\begin{equ}[eq:gen_a_priori]
|t^{-1}\E[F(X_t) - F(X_0)]|
\leq (2+\|u_0\|_\infty)^q\;.
\end{equ}
\end{lemma}

\begin{proof}
Throughout the proof, $q>0$ denotes a sufficiently large number depending only on $F$ and $P$, whose value may change from line to line.
% First, by Assumption \ref{ass:P}, the expectation $\E F(X_t)$ exists.

We can assume $\|\phi_i\|_{\CC^5}\leq 1$.
Equipping $\R^k$ with the norm $|z| = \max_{i=1,\ldots,k} |z_i|$,
we have $|X_t| \leq \|u_t\|_\infty$.
Therefore, by Taylor's theorem,
\begin{equs}[eq:F_u_expan]
F(X_t) - F(X_0)
&= \sum_{i=1}^k \partial_iF(X_0) \scal{u_t-u_0,\phi_i}
\\
& \quad + \frac{1}{2}\sum_{i,j=1}^k \partial_{i,j}^2 F(X_0) \scal{u_t-u_0,\phi_i}\scal{u_t-u_0,\phi_j}
\\
&\quad + \sum_{i=1}^k O(K|\scal{u_t-u_0,\phi_i}|^3)\;,
\end{equs}
where $K = \sup\{|F^{(3)}(x)|\,:\, |x| \leq \sup_{s\in[0,t]}\|u_s\|_\infty\}$ and $F^{(3)}$ is the third derivative of $F$.
Note that $|\partial_i F(X_0)| + |\partial_{i,j}^2 F(X_0)| \leq (2+\|u_0\|_\infty)^q$  since we assumed $F^{(3)}$ has polynomial growth.

We next compute the expectation of the right-hand side of \eqref{eq:F_u_expan}.
% First, by Markov's inequality and the fact that $\xi$ is Gaussian, for any $M>0$, we have $\P[Q_t^c] \lesssim t^M$ uniformly in $t\in (0,\frac12)$.
% In particular, $\P[Q_t] = 1-o(t)$.
We write
\begin{equ}
u_t-u_0
=
e^{t\Delta} u_0 - u_0
+ \int_0^t e^{(t-s)\Delta} P(u_s)\mrd s + \Psi_t
\end{equ}
where $\Psi_t = \int_0^t e^{(t-s)\Delta} \xi_s \mrd s$.
For the linear terms, since $\|e^{t\Delta}\phi_i - \phi_i - t\Delta\phi_i \|_\infty = O(t^2)$ due to $\phi_i \in \CC^5$,
\begin{equ}
\Gamma_{1,i}\eqdef \scal{
e^{t\Delta}u_0 - u_0,
\phi_i}
= t\scal{ u_0,\Delta \phi_i}
+ O(t^2 \|u_0\|_\infty)
\end{equ}
and
\begin{equs}
\Gamma_{2,i} \eqdef \scal{\int_0^t e^{(t-s)\Delta} P(u_s)\mrd s,\phi_i}
&=
\scal{ \int_0^t e^{(t-s)\Delta} P(u_0)\mrd s, \phi_i }
\\
&\qquad+ O\Big(t\sup_{s\in [0,t]}\|P(u_s)-P(u_0)\|_\infty\Big)\
\;.
\end{equs}
%Further, using $\|e^{r\Delta} f - f\|_{\CC^{-3}} \lesssim r\|f\|_\infty$ and again that $\phi_i\in\CC^5$,
%\begin{equ}
%scal{ \int_0^t e^{(t-s)\Delta} P(u_0)\mrd s, \phi_i }
%=
%\scal{ tP(u_0), \phi_i }
%+ O(t^2 \|P(u_0)\|_\infty)\;.
%\end{equ}
We also have
\begin{equs}
\scal{ \int_0^t e^{(t-s)\Delta} P(u_0)\mrd s, \phi_i }
&=
\scal{ P(u_0) , \int_0^t e^{(t-s)\Delta} \phi_i\mrd s  }\\
&=
\scal{ tP(u_0), \phi_i }
+
\scal{ P(u_0) , \int_0^t (e^{(t-s)\Delta} \phi_i - \phi_{i}) \mrd s  }\\
&=
\scal{ tP(u_0), \phi_i }
+ O(t^2 \|P(u_0)\|_\infty)\;.
\end{equs}
Moreover, since $\xi$ has zero expectation, $\E\scal{\Psi_t,\phi_i} = 0$.

Next, for the quadratic terms,
\begin{equs}
\E \scal{u_t-u_0,\phi_i}\scal{ u_t-u_0,\phi_j}
&=
\E \scal{\Psi_t , \phi_i} \scal{ \Psi_t , \phi_j} + \Gamma_{3}\;.
\end{equs}
where
\begin{equs}
\E \scal{\Psi_t , \phi_i} \scal{\Psi_t , \phi_j}
&=
\int_0^t \scal{\Cov_\xi * e^{(t-s)\Delta}\phi_i, e^{(t-s)\Delta}\phi_j}\mrd s
\\
&= t \scal{\Cov_\xi * \phi_i, \phi_j} + O(t^2)
\end{equs}
and the `error term' $\Gamma_3$, by Cauchy--Schwarz, is bounded above by a multiple of
\begin{equ}
\sum_{m=1}^k
\sqrt t \sqrt{\E(\Gamma_{1,m}^2 + \Gamma_{2,m}^2)}
+
\E(\Gamma_{1,m}^2 + \Gamma_{2,m}^2)\;.
\end{equ}
% Note that $\Gamma_{1,m} = O(t)$ and is deterministic.
To conclude the proof, it remains to bound the stochastic term $\sup_{s\in [0,t]}\|P(u_s)-P(u_0)\|_\infty$ in the definition of $\Gamma_{2,i}$
as well as the expectation of the error term $K |\scal{u_t-u_0,\phi_i}|^3$ in \eqref{eq:F_u_expan}.

Define $\tau \in (0,\frac12)$ by $\tau^{-1} = (2 + \|\xi\|_{\CC^{\alpha-2}(O)} + \|u_0\|_\infty)^{-1/\kappa}$,
where $\kappa,\alpha>0$ are as in Proposition \ref{prop:well-posed}.
Then, by Proposition \ref{prop:well-posed},
\begin{equ}
\sup_{s\in[0,\tau]}\|u_s\|_\infty
\leq
2\|u_0\| + 1\;,
\end{equ}
while by \eqref{eq:u_apriori} in Assumption \eqref{ass:P}
\begin{equ}
\sup_{s\in[\tau,1]}\|u_s\|_\infty \leq
(2+ \|u_0\|_\infty + \|\xi\|_{\CC^{\alpha-2}(O)} + \tau^{-1})^m \lesssim \tau^{-q}\;.
\end{equ}
By the moment bounds \eqref{eq:xi_moments}, for any $p\geq 1$, there exists $q>0$ such that
$
\E \tau^{-p} \leq (2 + \|u_0\|_\infty)^q
$,
so in conclusion,
\begin{equ}[eq:Pu_s_moment]
\E \sup_{s\in [0,1]} \|P(u_s)\|_\infty^2 \leq (2 + \|u_0\|_\infty)^q\;.
\end{equ}
In particular, by dominated convergence, as $t\downarrow0$,
\begin{equ}
\E \sup_{s\in [0,t]}\|P(u_s)-P(u_0)\|_\infty^2 \to 0
\end{equ}
and the left-hand side is bounded by $(2 + \|u_0\|_\infty)^q$ for all $t\in [0,1]$.

Finally, by similar reasoning, using the fact that $F^{(3)}$ has polynomial growth, we have $\E K |\scal{u_t-u_0,\phi_i}|^3 \leq t^{3/2}(2+\|u_0\|_\infty)^q$.
\end{proof}

\begin{lemma}\label{lem:gen_poly_classical}
Suppose Assumption \ref{ass:P} holds and that $\mu$ is a probability measure on $\CC(\T)$ that is invariant for the dynamic \eqref{eq:SPDE_classical}.
% i.e. $u_0\sim\mu$ implies $u_t\sim\mu$ for all $t\geq 0$.
Suppose $\E_{u_0\sim\mu}\|u_0\|^q_\infty<\infty$ for all $q \in [1,\infty)$.
Then, in the setting of Lemma \ref{lem:generator_classical},
\begin{equ}[eq:F_Euler_classical]
\E_{u_0\sim\mu} [
\mcL F(u_0)] 
=0\;.
\end{equ}
In particular, for $\phi\in\CC^\infty(\T)$ and integer $k \geq 1$,
\begin{equs}[eq:poly_Euler_classical]
\E_{u_0\sim\mu} \Big[ k\scal{ u_0, \phi }^{k-1}(\scal{ u_0, \Delta\phi }
&+ \scal{ P(u_0),\phi})
\\
&+ \frac{k(k-1)}{2} \scal{ u_0, \phi }^{k-2} \scal{\Cov_\xi *\phi, \phi} \Big] = 0\;,
\end{equs}
where the final term in the expectation is zero if $k=1$.
\end{lemma}

\begin{proof}
For the first claim,
\begin{equs}
\E_{u_0\sim \mu}[\mcL F(u_0)]
&= \E_{u_0\sim \mu}
\Big[
\lim_{t\downarrow0}t^{-1}\E_{\xi}[F(X_t) - F(X_0)]
\Big]
\\
&=
\lim_{t\downarrow0}\E_{u_0\sim\mu}[t^{-1}\E_{\xi}[F(X_t) - F(X_0)]
=0\;,
\end{equs}
where the first equality is due to \eqref{eq:gen_classical}, the second equality is due to \eqref{eq:gen_a_priori}  and dominated convergence,
and the final equality uses invariance of $\mu$ which implies $\E_{u_0\sim \mu} \E_{\xi} F(X_t) = \E_{u_0\sim\mu} F(X_0)$.
\eqref{eq:F_Euler_classical} follows by specialising to $F(x) = x^k$.
\end{proof}

\begin{remark}\label{rem:DS_eq}
Formally, for $\Cov_{\xi} = \delta$ (i.e. white noise case) the generator equation \eqref{eq:F_Euler_classical} can also be derived from the Dyson--Schwinger equations, i.e. integration by parts.
The latter in this context is
\begin{equ}
\E \nabla f(u) = \E f(u)(P(u) + \Delta u)
\end{equ}
for any (possibly field-valued) function $f(u)$ and where $\nabla f(u)$ is the gradient of $f$, namely $\scal{\nabla f(u),v} = Df(u)(v)$.
Taking $f = \nabla F(u)$ where $F(u) = F(X)$ and $X=(\scal{u,\phi_1},\ldots,\scal{u,\phi_k})$ and thus $\nabla F(u) = \sum_{i=1}^k \partial_i F(X)\phi_i$,
we obtain
\begin{equ}
\E \sum_{i,j} \partial_{i,j}^2 F(X) \phi_i\otimes\phi_j = \E \sum_i \partial_i F(X) \phi_i\otimes (P(u) + \Delta u)\;.
\end{equ}
We then obtain the above generator by evaluating the (tensor) field on the diagonal $x=y$ and integrating over $x$.
A similar derivation was noted in \cite[Sec.~3]{BHST87I}.
\end{remark}

\begin{theorem}\label{thm:non_G_classical}
Suppose that $P$ is a polynomial of odd degree with negative leading coefficient.
Then every invariant measure $\mu$ is not Gaussian.
\end{theorem}

\begin{example}
The prototypical example to which the theorem applies is
\begin{equ}
(\partial_t -\Delta) u = -u^p + \xi\;, \qquad u_0\in \CC(\T)\;,
\end{equ}
which has invariant measure
\begin{equ}
\mu(\mrd u) = Z^{-1} \exp\Big( -\frac{1}{p+1}\int u^{p+1} \Big) \nu(\mrd u)
\end{equ}
where $\nu$ is the law of the Gaussian free field on $\T$,
i.e. periodic Brownian motion.
\end{example}

Before giving the proof of Theorem~\ref{thm:non_G_classical}, we recall the definition of Hermite polynomials which will be used in this proof and also in Section~\ref{sec:DPD}. 

Given any $\sigma \in \R$, we define an associated sequence of Hermite polynomials $(H^{\sigma^2}_{n})_{n \in \N}$ on $\R$ inductively, setting $H_{0}^{\sigma^2}(x) = 1$ and, for $n > 0$, 
\begin{equ}\label{eq:def_hermite}
H^{\sigma^2}_{n}(x) =  x H^{\sigma^2}_{n-1}(x) - \sigma^2 \frac{\mathrm{d}}{\mathrm{d}x} H^{\sigma^2}_{n-1}(x)\;.
\end{equ}

\begin{proof}[of Theorem \ref{thm:non_G_classical}]
By Remark \ref{rem:assump_P}, our condition on $P$ implies both Assumption \ref{ass:P} and that $\sup_{u_0 \in\CC}\E \|u_1\|^q_\infty<\infty$ for all $q \in [1,\infty)$.
Taking $u_0\sim \mu$, we in particular have $\E_{u_0}\|u_0\|^q_\infty<\infty$.
Therefore, by \eqref{eq:poly_Euler_classical} of Lemma \ref{lem:gen_poly_classical}, for $\scal{\Cov_\xi*\phi,\phi}=1$ and $\phi \in \CC^5(\T)$,
\[
\E \Big[k\scal{u_0,\varphi}^{k-1}(\scal{u_0,\Delta\varphi} + \scal{P(u_0),\varphi}) + \frac{1}{2}k(k-1)\scal{u_0,\varphi}^{k-2}
\Big] = 0\;.
\]
Using the shorthand $X = \scal{u_0,\varphi}$, $Y = \scal{P(u_0),\varphi}$, and $Z = \scal{u_0,\Delta\varphi}$, we can rewrite the above as
\begin{equation}\label{eq:XY_Phi41}
\E{X^{k-1}(Y-Z)} = \frac{k-1}{2}\E{X^{k-2}}\;,\quad\forall k \geq 1\;,
\end{equation}
where the right-hand side is understood as zero for $k=1$.
It follows that, for every polynomial $Q$,
\begin{equ}
\E Q(X)(Y-Z) = \E Q'(X)/2\;.
\end{equ}
Hence $\E Q(X-\mu)(Y-Z) = \E Q'(X-\mu)/2$ for any $\mu\in\R$.
Taking $\mu=\E X$ and denoting $\hat X\eqdef X-\E X$, we obtain
\begin{equ}
\E Q(\hat X)(Y-Z) = \E Q'(\hat X)/2\;.
\end{equ}

Suppose now that $u_0$ is Gaussian with correlation function $\E{u_0(x)u_0(y)} = C(x,y)$.
%, which, by translation invariance, can be written as $C(x,y) = C(x-y)$.
Note that $\hat X$ is in the first homogeneous Wiener chaos.
We treat $C$ as a positive, bounded, symmetric operator on $L^2(\T)$ by $C\phi = \int C(x,\cdot)\phi(x) \mrd x$.
Since $\mu$ is a Gaussian measure on $\CC(\T)$, we have $\E \|u_0\|_\infty^2<\infty$, hence $\sup_{x,y}|C(x,y)|<\infty$ and thus $C$ is a bounded operator $L^1(\T)\to L^\infty(\T)$.

Taking $Q_k(x) = H^{\sigma^2}_k(x)$ with $\sigma^2 = \E \hat X^2$,
%the $k$-th Hermite polynomial with variance $\sigma^2 = \E \hat X^2$ and we have 
%$\sigma \ge 0$.
we claim that 
\begin{equ}[eq:X_Y_orth]
\E Q_k (\hat X) Y  = 0 \quad \forall k\geq 2\;.
\end{equ}
Indeed, recall the recursion
%$Q_{k}(x) =  x Q_{k-1}(x) - \sigma^2 Q'_{k-1}(x)$ from 
\eqref{eq:def_hermite}.
It follows that, if $\sigma^2 > 0$, then
\begin{equ}
\E Q_k(\hat X)(Y-Z) = \E Q_{k}'(\hat X)/2
\propto \E \hat X Q_{k}(\hat X) - \E Q_{k+1}(\hat X)\;.
\end{equ}
By orthogonality of Hermite polynomials, the right-hand side vanishes for $k \in \N \setminus \{1\}$.
Since $Z$ is in the 1st inhomogeneous Wiener chaos, we obtain \eqref{eq:X_Y_orth}.
If, on the other hand, $\sigma^2=0$, then $Q_k(\hat X) = 0$ almost surely for all $k\geq 1$, and we again obtain \eqref{eq:X_Y_orth}.

% For $k=1$, it follows directly from \eqref{eq:XY_Phi41} that
% \begin{equ}
% \E X(Y-Z) = \E\hat X(Y-Z) = \frac12\;.
% \end{equ}
Note that $Q_k (\hat X)$ has the Wiener-It\^o decomposition
\begin{equ}
Q_k (\hat X) = I_k (\phi^{\otimes k})
\end{equ}
where $I_k \colon H^{\otimes_{s} k} \to L^2(\P)$ is the Wiener isometry, $H$ is the Hilbert space given by completion of smooth functions with inner product $\scal{C \phi,\psi}$, and $H^{\otimes_{s} k}$ is $k$-fold symmetric tensor product of $H$. 

Recalling $Y=\int_{\T} P(u_0)(x)\phi(x)\mrd x$ and using our assumption that the leading term of $P(u)$ is $\theta u^p$ for $\theta \neq 0$ and $p$ odd, it readily follows that
\begin{equs}[eq:top_chaos]
0 &= \E Q_p(\hat X) Y
 = 
\theta 
\E
I_{p}( \phi^{ \otimes p} )
 I_{p} \Big(  \int_{\T}   \delta( \cdot - x)^{\otimes p} \phi(x) \mrd x \Big) \\ 
&= \theta 
\langle  \phi^{ \otimes p} , \int_{\T}   \delta( \cdot - x)^{\otimes p} \phi(x) \mrd x \rangle_{H^{\otimes_{s} p}}
= \theta \int_{\T}  (C\phi)(x)^{p} \phi(x) \mrd x\;.
\end{equs}
Above, in the second equality we used the fact that only the top chaos part of $Y$ contributes to the expectation which we then have written in terms of $I_{p}$,
and in the third equality we used the Wiener isometry property of $I_p$.

Consider now $\phi\in L^2(\T)$ a non-zero real eigenfunction of $C$
with eigenvalue $\lambda > 0$.
Take $\phi_n\in \CC^\infty(\T)$ such that $\phi_n\to\phi$ in $L^2$.
Then, since $C\colon L^2\to L^\infty$ is bounded,
we have $C\phi_n \to C\phi$ in $L^\infty$, hence
\begin{equ}[eq:eigenfunc_limit]
\scal{(C\phi_n)^p , \phi_n} \to \scal{ (C\phi)^p,\phi} = \lambda^p \scal{\phi^p,\phi}\;,
\end{equ}
which is strictly positive since $p$ is odd.
But, due to \eqref{eq:top_chaos} applied to $\phi_n$,
$\scal{(C\phi_n)^p , \phi_n} = 0$, which is a contradiction.
\end{proof}

\begin{remark}\label{rem:Phi31}
The argument above also works for the normalised $\Phi^3_1$ model
\begin{equ}
(\partial_t -\Delta) u = -u^2 - \Big(\int_{\T} u^2\Big) u + \xi\;, \qquad u_0\in \CC(\T)\;,
\end{equ}
which has invariant measure
\begin{equ}
\mu(\mrd u) = Z^{-1} \exp\Big( -\frac13\int_{\T} u^{3} - \frac14\Big(\int_{\T} u^2\Big)^2 \Big) \nu(\mrd u)
\end{equ}
where $\nu$ is the law of the Gaussian free field on $\T$.

More generally, we can allow $P(u)(x)$ in \eqref{eq:SPDE_classical} to be any polynomial with leading term
\begin{equ}[eq:general_P]
u(x)^\ell R\Big(\int_{\T} u^2,\ldots, \int_{\T} u^{2m}\Big)\;,
\end{equ}
where $\ell$ is odd and $R(y_1,\ldots, y_m)$ is a homogeneous polynomial which is non-zero whenever $y_1,\ldots, y_m > 0$.

In this case, the same argument would go through, except the integrand on the right-hand side of \eqref{eq:top_chaos}
would become
\[
(C\phi)(x)^{\ell} R\Big(\int_{\T} (C\phi)^2,\ldots, \int_{\T} (C\phi)^{2m} \Big) \phi(x)
\]
and one makes a corresponding change to \eqref{eq:eigenfunc_limit} so that the final term becomes
$\lambda^p \int_{\T}  \phi^{\ell+1} R\big(\int_{\T} \phi^2,\ldots, \int_{\T} \phi^{2m} \big) \mrd x$, which is non-zero by assumption.
\end{remark}

\section{The Da Prato--Debussche regime}\label{sec:DPD}

We now extend the argument in Section \ref{sec:classical} to the DPD regime for equations of the form
\begin{equ}\label{eq:SPDE}
(\partial_t +1 -\Delta ) u = P(u) + \xi\;,\quad u_0\in \CC^\alpha(\T^d)\;.
\end{equ}
Here and throughout this section, $P$ is a polynomial of degree $p\geq 1$ and $\xi$ is a stationary Gaussian noise on $\R \times \T^d$
with covariance $\delta\otimes \Cov_\xi \in \CD'(\R\times\T^d)$ (in particular white-in-time).
We fix in this section the degree $p\geq 1$ and a regularity parameter $\alpha<0$.

In this section, we will progressively add assumptions on the polynomial $P$ and the spatial covariance $\Cov_\xi$ (see Assumptions \ref{ass:covar_bound}, \ref{ass:P_DPD_global}, \ref{ass:P_DPD}, \ref{ass:moments}).
We keep these assumptions separate to emphasise which properties of $P$ and $\xi$ we are using at each step.
However, the reader should keep in mind that they are all satisfied for $P(x) = \sum_{j=0}^p a_j x^j$ where $p$ is odd and $a_p < 0$
and $|(1-\Delta)^{-1}*\Cov_\xi(x)| \lesssim |x|^{\rho}$ with $\rho> \max\{-d/p, 2\alpha\}$
for $\alpha>2/(1-2p)$
(see Appendix \ref{app:abs_cont}).
In particular, our results apply to the $\Phi^4_\delta$ model with $\delta<\frac{14}{5}$, see Example \ref{ex:Phi4}.

\begin{remark}
We consider $1-\Delta$ in \eqref{eq:SPDE} only for simplicity, and we could equally handle other elliptic operators like $(1-\Delta)^\sigma$ for $\sigma>0$.
\end{remark}

\subsection{Da Prato--Debussche argument}
\label{subsec:DPD_trick}

We first recall the classical argument of Da Prato--Debussche \cite{DPD02_SNS,DPD03_SQ}
for local well-posedness of stochastic PDEs of the form \eqref{eq:SPDE}.
There are several places where the argument is described, see e.g. \cite{Hairer14_ICM, Chandra_Weber_17_SPDEs, Mourrat_Weber_17_Phi42, Tsatsoulis_Weber_18_SG,Berglund_22_SPDEs, Chevyrev22_Hopf_lectures},
and we do not claim significant novelty in our presentation or results in this subsection.
(The only possible exception is Definition \ref{def:mfC} which identifies a good state space for the Markov process associated to our SPDE.)

At this stage, we remark that, if $d=2$ and $\Cov_{\xi}=\delta$ (the case considered in \cite{DPD03_SQ}), then by a Kolmogorov-type argument, $\xi$ is at best in $\CC^{\alpha-2}((-1,2)\times\T^d)$ for $\alpha<0$, hence, by parabolic regularity,
$u$ is at best in $\CC^{\alpha}([0,1]\times\T^d)$.
The function $P(u)$ is generally ill-defined on $\CC^\alpha$ for $\alpha<0$ as there is no canonical way
to extend the domain of the multiplication operator $(f,g)\mapsto fg$ from $\CC\times\CC$ to $\CC^\alpha\times\CC^\alpha$.
The equation \eqref{eq:SPDE} thus appears ill-posed.
Furthermore, replacing $\Cov_{\xi}$ by an approximation of the identity $\delta^\eps$ and sending the mollification scale $\eps\downarrow 0$ does not overcome this issue since, in general, the corresponding smooth solutions $u^\eps$ will fail to converge to a non-trivial limit.

The idea in \cite{DPD03_SQ} is to consider the solution to the equation
\begin{equ}[eq:SHE]
(\partial_t +1 -\Delta )\Psi = \xi
\end{equ}
and decompose $u = v+\Psi$.
Then the `remainder' $v$ solves the equation
\begin{equ}[eq:DPD_v_naive]
(\partial_t +1 -\Delta) v = P(v+\Psi)\;,
\end{equ}
where now we can hope to make sense of $P(v+\Psi)$ analytically provided that every power $\Psi^k$ is
replaced by another function $\Psi^{:k:}$ that we can define stochastically.

\subsubsection{Deterministic theory}\label{subsec:det_theory}

Recall that we fixed $\alpha < 0$.
Consider the Banach space of time-dependent distributions
\begin{equ}
\mfB = \bigoplus_{k=1}^p \CC([0,1],\CC^{\alpha k}(\T^d))\;.
\end{equ}
equipped with the norm
\begin{equ}
\|\bPsi\|_{\mfB} = \max_{k=1,\ldots,p} \sup_{t\in [0,1]}\|\Psi^{(k)}_t\|_{\CC^{\alpha k}(\T^d)}
\end{equ}
where we denote a generic element of $\mfB$ as $\bPsi = (\Psi^{(1)},\ldots,\Psi^{(p)})$.

Instead of \eqref{eq:DPD_v_naive}, we consider first the more general equation
\begin{equ}[eq:DPD_v_general]
(\partial_t +1 -\Delta) v = \sum_{n=0}^p\sum_{k=0}^n a_{n,k} v^{n-k}\Psi^{(k)}
\end{equ}
where $a_{n,k}\in\R$.
We suppose throughout the subsection that $\alpha(2p-1) > -2$, i.e.
\begin{equ}[eq:beta_assump]
2+\alpha p > -\alpha(p-1)\;.
\end{equ}
%(think: $\alpha p$ is regularity of $\Psi^{:p:}$, $2+\alpha p$ is maximum regularity of $v$, and $\alpha(p-1)$ is regularity of $\Psi^{:p-1:}$).
Fix furthermore
\begin{equ}[eq:gamma_def]
\gamma \in (-\alpha(p-1),2+\alpha p)\;.
\end{equ}
(Note that $\gamma \in (0,2)$.)
In particular, using that point-wise multiplication is a bounded bilinear map
$
\CC^{\varsigma_1}(\T^d) \times \CC^{\varsigma_2}(\T^d) \to \CC^{\varsigma_1 \wedge \varsigma_2}(\T^d)
$
as long as $\varsigma_1+\varsigma_2 > 0$, our condition on $\gamma,\alpha$ guarantees that
multiplication is continuous as a map
\begin{equ}[eq:mult]
\CC^\gamma(\T^d) \times \CC^{\alpha k}(\T^d) \to \CC^{\alpha k}(\T^d)
\end{equ}
for all $1\leq k< p$.
Consequently, the map
\begin{equ}
\CC^\gamma(\T^d)\times \CC^{\alpha k}(\T^d) \ni (v,\Psi^{(k)}) \mapsto v^{n-k}\Psi^{(k)} \in \CC^{\alpha k}(\T^d)\;,
\end{equ}
is locally Lipschitz for all $0\leq k\leq n\leq p$.
As a consequence,
% \eqref{eq:wick_Lip} (at each time $t\in[0,T]$)
one can use the smoothing properties of the heat flow and Banach's fixed point theorem to show the following
local well-posedness for \eqref{eq:DPD_v_general}
(see \cite[Prop.~4.4]{DPD03_SQ} and \cite[Thm.~3.3]{Tsatsoulis_Weber_18_SG} for similar statements).

\begin{proposition}\label{prop:local_well_posedness}
Consider $\gamma>0$ as in \eqref{eq:gamma_def}, $\eta\in( -\frac2p, \gamma]$, and $K > 2$. Denote
\begin{equ}
B_K \eqdef \{(v_0,\bPsi) \in \CC^{\eta}(\T^d)\times \mfB\,:\, \|v_0\|_{\CC^\eta(\T^d)} + \|\bPsi\|_{\mfB} \leq K\}\;.
\end{equ}
Then there exists $\kappa>0$, depending only on $\alpha,\eta,p,\gamma$, such that, for $T = K^{-1/\kappa} \in (0,1)$,
\eqref{eq:DPD_v_general} admits a unique solution in the Banach space $\mcS$ of functions $v\in\CC([0,T],\CC^{\eta}(\T^d))$
for which
\begin{equ}\label{def:S-norm}
\|v\|_{\mcS} \eqdef \sup_{t\in (0,T]} \{\|v_t\|_{\CC^\eta(\T^d)} + t^{-\frac{\eta\wedge0}{2}}\|v_t\|_\infty + t^{\frac{\gamma-\eta}{2}} \|v_t\|_{\CC^\gamma(\T^d)}\} <\infty\;.
\end{equ}
Moreover, the solution map
\begin{equ}
B_K \ni (v_0,\bPsi) \mapsto v \in \mcS
\end{equ}
is Lipschitz and $\|v\|_{\mcS} \leq C(\|v_0\|_{\CC^\eta} + 1)$, where $C$ depends only on $\alpha,\eta,p,\gamma$.
\end{proposition}
\begin{remark}
The threshold $\eta>-\frac2p$ is sharp and is necessary even for local well-posedness of $(\partial_t +1 -\Delta)v = \pm v^p$ with $v_0\in\CC^{\eta}(\T^d)$,
see \cite{COW_22_norm_inf,Chevyrev22_norm_inf}.
\end{remark}
We now return to our problem \eqref{eq:DPD_v_naive} and set $a_{n,k}$ such that
\begin{equ}[eq:P_expan]
P(x+y) = \sum_{n=0}^p\sum_{k=0}^n a_{n,k} x^{n-k}y^k\;.
\end{equ}
In this case, if $\Psi^{(k)}=\Psi^k$, then \eqref{eq:DPD_v_general} agrees with \eqref{eq:DPD_v_naive}.
However, for $\Psi\in\CC^\alpha$ for $\alpha<0$, $\Psi^k$ is generically not well-defined for $k\geq 2$, which is the whole reason for introducing the general equation \eqref{eq:DPD_v_general}.

We next introduce a set of distributions on which `Wick powers' $\Psi^{:k:}$ are well-defined.
Fix a mollifier $\chi \in \CC^\infty_c(\R^d)$ with $\int\chi = 1$,
and write $\chi^\eps(x) = \eps^{-d}\chi(x/\eps)$ and $\Psi^\eps = \Psi*\chi^\eps$.

\begin{definition}\label{def:mfA}
Let $\mfA \subset \CC([0,1],\CC^\alpha(\T^d))$ be the subset of all time-dependent distributions $\Psi$ such that the collection of distributions
\begin{equ}[eq:wick_power_def]
\bPsi
= (\Psi,\Psi^{:2:},\ldots,\Psi^{:p:})
\eqdef
\lim_{\eps\downarrow 0} \big\{ H^{\sigma^2_\eps}_k(\Psi^\eps)
\big\}_{k=1,\ldots, p}
 \in \mfB
\end{equ}
exists as a limit in $\mfB$,
where 
\begin{equ}[eq:variance]
\sigma^2_\eps = \E_{\Phi\sim\nu}[ \Phi^\eps(0)^2] = \scal{\Cov_{\xi} *(1-\Delta)^{-1}*\chi^\eps,\chi^\eps}
\end{equ}
where $\nu$ is the law of the Gaussian field with covariance $\Cov_{\xi}*(1-\Delta)^{-1}$.
\end{definition}

\begin{remark}
The set $\mfA$ is measurable (this follows, e.g. from the countable criterion for negative H\"older spaces of \cite[Sec.~12]{Caravenna_Zambotti_20_Reconstruct} or \cite[Prop.~3.20]{Hairer14}).
Furthermore, the map $\Psi\mapsto \Psi^{:k:}$ is measurable.

It is not entirely trivial, however, that $\mfA$ is non-empty. In fact, if $\lim_{\eps\downarrow 0}\sigma_\eps^2 = \infty$,
which is the case of interest in this section,
then the intersection of $\mfA$ with, say, $\CC([0,1],L^\infty(\T^d))$ is clearly empty.
Propositions \ref{prop:wick_GFF} and \ref{prop:Psi0_mfC} will provide us with example distributions in $\mfA$.
% and explain its relevance in the solution theory of \eqref{eq:SPDE}.
\end{remark}

Recall the binomial identity for Hermite polynomials
\begin{equ}[eq:binomial]
H^{\sigma^2}_k(x+y) = \sum_{j=0}^k \binom kj x^j H^{\sigma^2}_{k-j}(y)\;.
\end{equ}
We further recall assumption \eqref{eq:beta_assump} and a choice for $\gamma$ as in \eqref{eq:gamma_def}.
The following result shows that $\mfA$ is closed under addition of sufficiently regular functions.

\begin{lemma}\label{lem:v_Psi}
If $\Psi\in\mfA$ and $v\in \CC([0,1],\CC^\gamma(\T^d))$,
then $\Psi+v\in\mfA$ and
\begin{equ}[eq:binomial_v_Psi]
(v+\Psi)^{:k:} = \sum_{j=0}^k \binom{k}{j} v^j \Psi^{:k-j:}\;.
\end{equ}
Moreover, for $\Psi$ fixed,
\begin{equ}
\CC([0,1],\CC^\gamma(\T^d)) \ni v \mapsto (v+\Psi)^{:k:} \in \CC([0,1],\CC^{\alpha k}(\T^d))
\end{equ}
is locally Lipschitz and polynomial of degree $k$.
\end{lemma}

\begin{proof}
This follows from boundedness of multiplication \eqref{eq:mult}, the fact that $\CC^\gamma(\T^d)$ is a Banach algebra,
and the binomial identity \eqref{eq:binomial}.
% the limit in \eqref{eq:u_Wick} exists, depends only on $v+\Psi$, and defines a measurable map
% \begin{equ}
% \CC([0,1],\CC^\gamma(\T^d)) + \mfA \ni v+\Psi \mapsto (v+\Psi)^{:k:} \in \CC([0,1],\CC^{\alpha k}
% (\T^d))
% \end{equ}
% for which
% $(v+\Psi)^{:k:} = \sum_{j=0}^k \binom{k}{j}v^j\Psi^{:k-j:}$.
\end{proof}

\begin{remark}\label{rem:gamma_upper_bound}
The upper bound $\gamma<2+\alpha p$ in \eqref{eq:gamma_def}, which is assumed throughout this subsection, is not needed for Lemma \ref{lem:v_Psi}.
\end{remark}

% We extend the definition of Wick powers to the space $\CC([0,1],\CC^\gamma(\T^d)) + \mfA$, by
% \ilya{added time variables here for simplicity}
% \begin{equ}[eq:u_Wick]
% (v_t+\Psi_t)^{:k:} = \lim_{\eps\downarrow 0} H^{\sigma^2_\eps}_k (v^\eps_t + \Psi^\eps_t)\;.
% \end{equ}
% if $\bar v + \bar\Psi = v+\Psi \in \CC^\gamma(\T^d) + \mfA$, then
% \begin{equ}
% \sum_{j=0}^k \binom{k}{j}v^j\Psi^{:k-j:}
% =
% \sum_{j=0}^k \binom{k}{j}\bar v^j \bar \Psi^{:k-j:}\;.
% \end{equ}
% if $v = \bar\Psi-\Psi\in\CC^\gamma(\T^d)$ for $\Psi,\bar\Psi\in \mfA$, then
% \begin{equ}
% \bar \Psi^{:k:} = (\Psi + v)^{:k:}\;.
% \end{equ} 

For $\Psi \in \mfA$ and a polynomial of one variable $P(t) = \sum_{j=0}^{p} a_{j} t^j$, we adopt the standard notation
\begin{equ}
\Wick{P(\Psi)} = \sum_{j=0}^{p} a_{j} \Psi^{:j:}
\end{equ}
for the corresponding local Wick polynomial.

For $\eta$ as in Proposition \ref{prop:local_well_posedness} and $\Psi \in \mfA$, we now consider the `remainder equation'
\begin{equ}[eq:DPD_v]
(\partial_t +1 -\Delta) v = \Wick{P(v+\Psi)}\;,\quad v_0\in\CC^{\eta}(\T^d)\;,
\end{equ}
which, due to the binomial-type identity \eqref{eq:binomial_v_Psi}, is a special case of \eqref{eq:DPD_v_general}.
Proposition \ref{prop:local_well_posedness}
therefore implies that existence and uniqueness of local-in-time solutions for $v$ in the Banach space $\mcS$.
% We then \emph{define} $u = v+ \Psi$ as the solution to~\eqref{eq:SPDE} with initial condition $u_0 \eqdef v_0 + \Psi_0$.

\begin{remark}\label{rem:renormalised_PDE}
By the local Lipschitz continuity of Proposition \ref{prop:local_well_posedness}
in both the $\bPsi$ variable and the initial condition, 
the function $u = v+\Psi$
%(which is what we would like to call the solution to~\eqref{eq:SPDE})
satisfies $u=\lim_{\eps\downarrow0}u^\eps$,
where $u^\eps$ solves the `renormalised' PDE
\begin{equ}[eq:renormalised_PDE]
(\partial_t + 1 -\Delta) u^\eps = \sum_{j=0}^p a_j H_j^{\sigma_\eps^2}(u^\eps) + \xi*\chi^\eps\;,\quad u^\eps_0 = v_0^\eps + \Psi_0^\eps\;,
\end{equ}
where $\xi \eqdef (\partial_t +1-\Delta)\Psi$.
\end{remark}

\subsubsection{Probabilistic theory}

Our interest in $\mfA$ comes from the fact that the random distribution $\Psi$ given as the solution to $(\partial_t +1-\Delta)\Psi = \xi$ with initial condition $\Psi_0\sim \nu$ for $\nu$ as in \eqref{eq:variance} almost surely takes values in $\mfA$,
i.e. the Wick powers of $\Psi$ are well defined and behave as required in Definition \ref{def:mfA}.
We refer to the survey of Da Prato--Tubaro \cite{Da_Prato_Tubaro_23_Wick} for a more in-depth introduction to Wick powers.

Recall that $\alpha<0$ is fixed.
As a warm up, we have the following result for the initial condition $\Psi_0$.
Recall that $\nu$ is the Gaussian field with covariance $\Cov_\nu = (1-\Delta)^{-1}*\Cov_{\xi}$.

\begin{assumption}\label{ass:covar_bound}
There exists $\rho > \max\{-d/p,2\alpha\}$ such that
$|\Cov_\nu(x)|\lesssim |x|^{\rho}$ uniformly in $x\in\T^d\setminus\{0\}$.
\end{assumption}

\begin{proposition}\label{prop:wick_GFF}
Suppose Assumption \ref{ass:covar_bound} holds and
denote $\Psi_0\sim\nu$.
Then for all integers $k=1,\ldots,p$
the limit $\lim_{\eps\downarrow 0} H^{\sigma^2_\eps}_k(\Psi^\eps_0)$
exists in $\CC^{\alpha k}(\T^d)$ in $L^q(\nu)$ for all $q \in [1,\infty)$ and $\nu$-almost surely.
\end{proposition}

\begin{proof}
This is standard, so we only sketch the idea.
Observe the bound
\begin{equ}[eq:GFF_Wick]
\E \scal{\Psi^{:k:}_0,\phi_x^\lambda}^2 = \int_{\T^{2d}} \mrd y \mrd z \phi_x^\lambda(y)\phi_x^\lambda(z) \Cov_{\nu}(y-z)^k \lesssim \lambda^{\rho k}\;,
\end{equ}
where we used $\Cov_{\nu}(y-z)^k \lesssim |y-z|^{\rho k}$ and $\rho k > -d$ by Assumption \ref{ass:covar_bound}.
The approximations $H^{\sigma^2_\eps}_k(\Psi^\eps_0)$ satisfy the same bound
in which we replace $\Cov_{\nu}$ by $\Cov_{\nu}*\chi^\eps*\chi^\eps(-\cdot)$.
Since $\alpha<\rho/2$, by equivalence of Gaussian moments and by a Kolmogorov-type argument, one can show that there exists $\kappa>0$ sufficiently small such that, for all $q\in [1,\infty)$, uniformly in $\eps,\bar\eps\in (0,1)$,
\begin{equ}
\E
\Big[
\|H^{\sigma^2_\eps}_k(\Psi^\eps_0) - H^{\sigma^2_{\bar\eps}}_k(\Psi^{\bar\eps}_0) \|_{\CC^{\alpha k}}^q \Big]^{1/q}\lesssim |\eps - \bar\eps|^\kappa\;.
\end{equ}
Applying Kolmogorov's theorem to the mollification variable $\eps$,
we conclude that, for all $q\in [1,\infty)$,
\begin{equ}
\E\Big[
\Big(\sup_{0<\eps < \bar\eps < 1}|\eps-\bar\eps|^{-\kappa/2} \|H^{\sigma^2_\eps}_k(\Psi^\eps_0) - H^{\sigma^2_{\bar\eps}}_k(\Psi^{\bar\eps}_0)
\|_{\CC^{\alpha k}}
\Big)^q
\Big] < \infty\;.
\end{equ}
This shows convergence of $H^{\sigma^2_\eps}_k(\Psi^\eps_0)$ in $\CC^{\alpha k}$ in $L^q(\nu)$ and almost surely.
\end{proof}

% Proposition \ref{prop:wick_GFF} provides us with many examples of functions in $\mfA$; in fact,
% a standard result in Gaussian analysis is that the support of $\nu$\ilya{double check - do we need assumption on $\Cov_\xi$ for this?} is all of $\CC^\alpha(\T^d)$,
% and therefore the set of all $\Psi_0\in\CC^{\alpha}(\T^d)$ such that $\lim_{\eps\downarrow 0} H^{\sigma^2_\eps}_k(\Psi^\eps_0)$
% exists in $\CC^{\alpha k}(\T^d)$ is dense in $\CC^\alpha(\T^d)$.

For the local solution theory of the \emph{stochastic} PDE \eqref{eq:SPDE},
we need to improve Proposition \ref{prop:wick_GFF} and show the Wick powers of the linear solution $\Psi$ from \eqref{eq:SHE}
takes values in $\mfA$.
For our purposes, we crucially need the initial condition of $\Psi$ to be \emph{any}\footnote{This in contrast to most other presentations which take the initial condition as $0$ or take $\Psi$ as the stationary solution to $(\partial_t +1-\Delta)\Psi = \xi$, see e.g. \cite{Tsatsoulis_Weber_18_SG}.} element of a deterministic set $\mfC\subset\CC^\alpha(\T^d)$ which supports $\nu$.
The following definition of $\mfC$ is the only result of this subsection that we believe is new.

\begin{definition}\label{def:mfC}
Denote by $\mfC$ the set of all $\Psi_0\in\CC^\alpha(\T^d)$ such that
% \begin{equ}
% \fancynorm{\Psi_0} \eqdef \sup_{x,\lambda,\phi}\sup_{k=0,\ldots, p}\sup_{0<s<t<1} \lambda^{-2\rho+\kappa}|t-s|^{-\kappa/2}\E_{\tilde\Psi} |\scal{Z_{t;k} -Z_{s;k},\phi^\lambda_x}|^2 < \infty
% \end{equ}
the stochastic process $\{H_k^{\sigma_\eps^2}(\Psi^\eps_t) \}_{t\in[0,1]}$ for $k=1,\ldots, p$, converges in $\CC([0,1],\CC^{\alpha k}(\T^d))$ in $L^q(\P_\xi)$ for every $q\in [1,\infty)$ and $\P_\xi$-almost surely as $\eps\downarrow0$.

Above, as usual, $\Psi^\eps=\Psi*\chi^\eps$ and the stochastic process $\{\Psi_t\}_{t\in [0,1]}$ solves the linear stochastic heat equation $(\partial_t+1-\Delta)\Psi=\xi$ with initial condition $\Psi_0$,
and $\xi$, as in \eqref{eq:SPDE}, is Gaussian on $\R \times \T^d$ with covariance $\delta\otimes \Cov_{\xi}$.
We denote $\Psi^{:k:} = \lim_{\eps\downarrow0}H_k^{\sigma_\eps^2}(\Psi^\eps)$ and 
\begin{equ}[eq:C_norm]
\fancynorm{\Psi_0}_q \eqdef \max_{k=1,\ldots, p} \{ \E_\xi \|\Psi^{:k:}\|^q_{\CC([0,1],\CC^{\alpha k}(\T^d))}\}^{1/q}
\end{equ}
and, as earlier, $\bPsi = (\Psi,\Psi^{:2:},\ldots,\Psi^{:p:})$.
\end{definition}
Elements of $\mfC$ will play the role of sufficiently nice initial conditions for the SPDE \eqref{eq:SPDE}
for which the generator exists (see Lemma \ref{lem:generator}).
We first show that $\nu(\mfC)=1$.
% In light of Proposition \ref{prop:wick_GFF} and the need for assumption \eqref{eq:beta_assump}, it is natural to make the following assumption on $\Cov_{\xi}$.
% %
% \begin{assumption}\label{ass:covar_bound}
% $|((1-\Delta)^{-1}*\Cov_{\xi})(x,y)|\lesssim |x-y|^{\rho}$ for some $\rho\in (-d/p,0]$
% such that $\rho(p-1/2) > -2$. 
% \ilya{move earlier and split}
% \end{assumption}
%
%Sanity check: if $d=2$ and $\bar\CC=\delta$, then  $(-\Delta+1)^{-1}(x,y) \asymp |\log|x-y||$, so the bound holds for all $\rho<0$.
%If $d=3$ and $\bar\CC=\delta$, then $(-\Delta+1)^{-1}(x,y)\asymp |x-y|^{-1}$, so $\rho\in (-d/p,0]$ already does not hold for $p=3$ as expected. The second condition is equivalent to $5\rho/2 > -2$ i.e. $\rho > -4/5$ i.e. $\alpha > -2/5$ which is precisely the DPD threshold in $d=3$.
%
%
\begin{proposition}\label{prop:Psi0_mfC}
Suppose Assumption \ref{ass:covar_bound} holds and
let $\Psi_0\sim \nu$.
Then $\nu(\mfC)=1$ and
$\E \fancynorm{\Psi_0}_q^q < \infty$ for all $q\in [1,\infty)$.
\end{proposition}

\begin{proof}
Denoting $Z^\eps_t = H_k^{\sigma_\eps^2}(\Psi^\eps_t)$ and $\E = \E_{\Psi_0\sim\nu} \E_\xi$,
it suffices to prove that there exists $\kappa>0$ such that, for any $q\in [1,\infty)$, uniformly in $s,t\in [0,1],
\eps,\bar\eps\in (0,1)$,
\begin{equ}[eq:EZ]
(\E \|Z^\eps_{s} - Z^{\bar\eps}_t\|_{\CC^{\alpha k}(\T^d)}^q )^{1/q}
\lesssim |t-s|^{\kappa/2} + |\eps-\bar\eps|^\kappa\;.
\end{equ}
Indeed, denoting $Z^\eps_{s,t} = Z^\eps_t-Z^\eps_s$,
\eqref{eq:EZ} implies that
\begin{equ}
(\E \|Z^\eps_{s,t} - Z^{\bar\eps}_{s,t}\|_{\CC^{\alpha k}(\T^d)}^q )^{1/q} \lesssim
|t-s|^{\kappa/4}|\eps-\bar\eps|^{\kappa/2}
\end{equ}
which in turn implies, by Kolmogorov's theorem (first applied to the time variable $t$ and then to the mollification scale $\eps$),
denoting $\mcE = \CC^{\kappa/8}([0,1],\CC^{\alpha k}(\T^d))$,
\begin{equ}
\E\Big[
\Big(\sup_{0<\eps < \bar\eps < 1}|\eps-\bar\eps|^{-\kappa/4} \|Z^\eps - Z^{\bar\eps}
\|_{\mcE}
\Big)^q
\Big] < \infty\;.
\end{equ}
In particular, sampling first $\Psi_0$, we have
 $\E_\xi\sup_{\eps,\bar\eps}(|\eps-\bar\eps|^{-\kappa/4}\|Z^\eps-Z^{\bar\eps}\|)^q = \E[\sup_{\eps,\bar\eps}(|\eps-\bar\eps|^{-\kappa/4}\|Z^\eps - Z^{\bar\eps}
\|_{\mcE})^q|\Psi_0] \leq C$ where $C$ is a random variable in $L^1(\nu)$, which implies that, for $\nu$-almost every $\Psi_0$, the limit $\lim_{\eps\downarrow0}Z^\eps$ exists in $\mcE$ in $L^q(\P_\xi)$ and $\P_\xi$-a.s.,
and this limit is in $L^q(\nu)$.

To prove \eqref{eq:EZ}, by equivalence of Gaussian moments and by decomposing a test function $\phi^\lambda_x$ into a wavelet basis as in \cite[Prop.~3.20]{Hairer14}, it furthermore suffices to show that
\begin{equ}
\E|\scal{Z_{t}^{\eps} -Z_{s}^{\bar\eps},\phi^\lambda_x}|^2 \lesssim \lambda^{k\rho-\kappa}(|t-s|^{\kappa/2} + |\eps-\bar\eps|^\kappa)
\end{equ}
uniformly over $x,\lambda,\phi,s,t,\eps,\bar\eps$ - see the proof of \cite[Prop.~9.5]{Hairer14} for more detail.
However, it is simple to verify that the (spatial) covariance of $\Psi_t$ is the same as that of $\Psi_0$.
Therefore $Z_{t}^\eps$ is the $k$-th Wick power of the stationary solution of the SHE driven by $\xi^\eps$,
and the above bound follows from an identity similar to \eqref{eq:GFF_Wick}.
\end{proof} 
Next, we show that $\mfC$ is stable under addition of sufficient regular functions.

\begin{lemma}\label{lem:Psi_plus_v}
Consider $\Psi_0\in \mfC$ and $v_0\in\CC^\gamma(\T^d)$ with $\gamma > -\alpha(p-1)$.
Then $\Psi_0+v_0 \in \mfC$ with $\fancynorm{\Psi_0+v_0}_q \lesssim (1+\fancynorm{\Psi_0}_q) (1 + \|v_0\|_{\CC^\gamma}^{p})$.
\end{lemma}
\begin{proof}
This follows from Lemma \ref{lem:v_Psi} applied to $\Psi$ and the function $v_t = e^{t(\Delta-1)}v_0$, which is in $\CC([0,1],\CC^\beta(\T^d))$ for any $\beta\in (-\alpha(p-1),\gamma)$ (this $\beta$ takes the role of $\gamma$ in Lemma \ref{lem:v_Psi}; recall that no upper bound on $\gamma$ is required due to Remark \ref{rem:gamma_upper_bound}).
\end{proof}

For $\Psi_0\in\mfC$, by definition, $\Psi \in \mfA$ (i.e. $\bPsi \in \mfB$) almost surely.
Therefore, if \eqref{eq:beta_assump} holds,
then we can apply Proposition \ref{prop:local_well_posedness}
to obtain local-in-time (random) solutions $v$ to \eqref{eq:DPD_v}.
We next consider an assumption that these solutions are \emph{global}.

\begin{assumption}\label{ass:P_DPD_global}
\eqref{eq:beta_assump} holds and,
for $\eta,\gamma$ as in Proposition \ref{prop:local_well_posedness} and all $\Psi\in \mfA$ and $v_0\in\CC^\eta(\T^d)$, the equation \eqref{eq:DPD_v}
has a unique solution in $\CC((0,1],\CC^\gamma)$.
\end{assumption}
Since $\alpha>-\frac2p$ by \eqref{eq:beta_assump}, note that $\eta=\alpha$ satisfies the conditions of Proposition \ref{prop:local_well_posedness}.

% \begin{remark}
% Depending on the non-linearity $P$, there are 
% several ways to verify Assumption \ref{ass:P_DPD_global}.
% \end{remark}

\begin{definition}[Markov process]\label{def:Markov}
Suppose Assumption \ref{ass:P_DPD_global} holds.
Consider $\gamma>0$.
We define $u_t \in \CC^\alpha$ with initial condition $u_0\in \CC^{\alpha}$ as follows.
Fix any $\Psi_0\in\mfC$.
Let $v\in \CC((0,1],\CC^\gamma)$ be the solution to \eqref{eq:DPD_v} with initial condition $v_0 = u_0-\Psi_0\in\CC^{\alpha}$.
We then define $u_t = v_t+\Psi_t \in\CC^\alpha$ for $t\in [0,1]$.
\end{definition}

It may appear unnatural to make an arbitrary choice of $\Psi_0\in\mfC$ in the definition of $u$.
The next result implies that a different choice leads to the same definition of $u$.

\begin{proposition}\label{prop:Markov_welldefined}
Suppose we are in the setting of Definition \ref{def:Markov}.
Consider $\bar\Psi_0\in\mfC$ and let $\bar v$ solve
\begin{equ}
(\partial_t +1 -\Delta)\bar v = \Wick{P(\bar v+\bar \Psi)}\;,\quad \bar v_0 = u_0 - \bar \Psi_0 \in \CC^{\alpha}(\T^d)\;,
\end{equ}
where $\bar\Psi$ solves $(\partial_t+1-\Delta)\bar\Psi = \xi$ with initial condition $\bar\Psi_0$ and where
$\xi$ is the same noise appearing in the definition of $\Psi$.
Then, almost surely, $\bar v_t + \bar\Psi_t = u_t$ for all $t\in[0,1]$.
\end{proposition}

\begin{proof}
This follows from Remark \ref{rem:renormalised_PDE}, which implies that $u$ and $\bar v+\bar\Psi$ are both limits of \eqref{eq:renormalised_PDE}.
\end{proof}

The following result now shows that the solution $u$ to \eqref{eq:SPDE} (in the sense of Definition \ref{def:Markov}) takes values in $\mfC$.

\begin{lemma}\label{lem:u_in_mfC}
Suppose we are in the setting of Definition \ref{def:Markov}
and that Assumption \ref{ass:covar_bound} holds.
Then, for every $t\in(0,1]$, $u_t \in \mfC$ almost surely.
\end{lemma}

\begin{proof}
Let $\Psi$ and $v$ be as in Definition \ref{def:Markov}
and fix $t\in (0,1]$.
Then, almost surely, $v_t\in \CC^\gamma(\T^d)$.
Furthermore, we write $\Psi_t = e^{t(\Delta-1)} \Psi_0 + \int_0^t e^{(t-s)(\Delta-1)}\xi_s\mrd s$
wherein $e^{t(\Delta-1)} \Psi_0 \in \CC^\infty(\T^d)$.
We also have $\int_0^t e^{(t-s)(\Delta-1)}\xi_s\mrd s \eqlaw \Phi + f$,
where $\Phi \sim \nu$ and $f$ is almost surely in $\CC^\infty(\T^d)$.
By Proposition \ref{prop:Psi0_mfC}, $\Phi$ is almost surely in $\mfC$.
The conclusion now follows from Lemma \ref{lem:Psi_plus_v}.
\end{proof}

\begin{remark}
Propositions~\ref{prop:local_well_posedness} and~\ref{prop:Markov_welldefined} allow us to solve the $\Phi^4_{\delta}$ SPDE for $\delta < 14/5$ \eqref{eq:SPDE} by constructing Wick powers of the linear solution $\Psi$ and deriving a well-posed equation for a remainder $v$ with respect to this linear solution $\Psi$. 

A conceptually similar but more involved argument allows us to go beyond $\delta < 14/5$ - if we 
%which we do not consider here 
solved for the higher order remainder $w_t = v_t - \int_0^t e^{(t-s)\Delta}\Psi^{:p:}_s\mrd s$ then we would obtain a solution theory for a broader range of parameters, e.g. $\Phi^4_{\delta}$ for $\delta<3$.
However, these sorts of methods cannot be pushed forward in a straightforward way to cross the $\delta = 3$ threshold. 
\end{remark}

\subsection{The Markov process}
\label{sec:Markov}

By Lemma \ref{lem:u_in_mfC}, $u$ from Definition \ref{def:Markov} defines a Markov process on the set $\mfC\subset \CC^\alpha(\T^d)$.
Furthermore, for $u_0\in\mfC$, we may simply take $\Psi_0 = u_0$ and thus $v_0=0$ in Definition \ref{def:Markov}.
It is therefore natural to impose the following assumption on $P$, which is an analogue of Assumption \ref{ass:P}.

\begin{assumption}\label{ass:P_DPD}
Assumption \ref{ass:P_DPD_global} holds and there exists $m>0$ such that, if $\Psi\in\mfA$ and $v$ is the solution to \eqref{eq:DPD_v} with $v_0=0$,
then for all $s\in (0,1]$,
\begin{equ}[eq:u_apriori_DPD]
\|v_s\|_{\CC^\gamma(\T^d)} \leq (2+ \|\bPsi\|_{\mfB} + s^{-1})^m\;.
\end{equ}
\end{assumption}
The following is an analogue of Lemma \ref{lem:generator_classical}.
\begin{lemma}\label{lem:generator}
Suppose Assumption \ref{ass:P_DPD} holds.
Let $u_0 \in \mfC$.
Consider $\phi_1,\ldots, \phi_k \in \CC^\infty(\T^d)$ and $F \in \CC^3(\R^k)$ with $F^{(3)}$ of polynomial growth.
Denote
\begin{equ}
X_t \eqdef (\scal{ u_t,\phi_1}, \ldots, \scal{u_t,\phi_k}) \in \R^k\;.
\end{equ}
Then
\begin{equs}
\lim_{t\to 0}t^{-1}\E[F(X_t) - F(X_0)]
&= \sum_{i=1}^k \partial_i F(X_0)(\scal{ u_0, (\Delta-1)\phi_i } + \scal{\Wick{P(u_0)},\phi_i})
\\
&\quad + \frac{1}{2}\sum_{i,j=1}^k \partial_{i,j} F(X_0) \scal{\Cov_{\xi}*\phi_i,\phi_j}\;.
\end{equs}
and there exists $q>0$, depending only on $P$ and $F$, such that, for all $t\in (0,\frac12)$,
\begin{equ}
|t^{-1}\E[F(X_t) - F(X_0)]| \leq (2+\fancynorm{u_0}_q)^q\;.
\end{equ}
\end{lemma}

\begin{proof}
We let $q>0$ denote a sufficiently large number, depending only on $F$ and $P$, whose value may change from line to line.
We write
\begin{equ}
u_t-u_0
= (e^{t(\Delta-1)}u_0 - u_0)
+\int_0^t e^{(t-s)(\Delta-1)} \Wick{P( u_s )}\mrd s
+ \int_0^t e^{(t-s)(\Delta-1)} \xi_s \mrd s\;.
\end{equ}
Arguing like in the proof of Lemma \ref{lem:generator_classical},
for $\phi\in\CC^\infty(\T^d)$ we have
\begin{equ}
\scal{e^{t(\Delta-1)}u_0 - u_0,\phi} = t\scal{u_0,(\Delta-1)\phi} + O(t^2\|u_0\|_{\CC^{\alpha}})
\end{equ}
and
\begin{equs}{}
&\Big|
\scal{\int_0^t e^{(t-s)(\Delta-1)} \Wick{P( u_s )} \mrd s,\phi} - t\scal{ \Wick{P(u_0)},\phi}
\Big|
\\
&\qquad \lesssim t^2\|\Wick{P(u_0)}\|_{\CC^{\alpha p}(\T^d)}
+
t\sup_{s\in [0,t]} \|\Wick{P(u_s)} - \Wick{P(u_0)}\|_{\CC^{\alpha p}(\T^d)}\;.
\end{equs}
Let $\Psi_0=u_0$ and $v,\Psi$ be as in Definition \ref{def:Markov}, so that $u=v+\Psi$. 
By continuity of multiplication
\begin{equ}[eq:cont_mult]
\CC^\gamma(\T^d)\times \CC^{\alpha j}(\T^d) \ni(w,\psi) \mapsto w\psi \in\CC^{\alpha j}(\T^d)
\end{equ}
and almost sure continuity of $[0,1] \ni t\mapsto (v_t,\Psi^{:j:}_t) \in \CC^\gamma(\T^d)\times \CC^{\alpha j}(\T^d)$
for $j=0,\ldots,p$,
we have $\sup_{s\in [0,t]}\|\Wick{P(u_s)} - \Wick{P(u_0)}\|_{\CC^{\alpha p}(\T^d)} \to 0$ almost surely.

We next bound the moments of $\sup_{s\in [0,t]} \|\Wick{P(u_s)} \|_{\CC^{\alpha p}(\T^d)}$.
Applying Proposition \ref{prop:local_well_posedness} with $\gamma = \eta$,
there exists $\tau \in (0,\frac12]$ such that $\tau^{-1} \leq (2+\|\bPsi\|_{\mfB})^{q}$
%, and thus $\mcS = \CC([0,\tau],\CC^\gamma(\T^d))$,
and $\sup_{s\in [0,\tau]}\|v_s\|_{\CC^\gamma(\T^d)} \leq C$.
Moreover, by \eqref{eq:u_apriori_DPD} of Assumption \ref{ass:P_DPD}, we have
\begin{equ}
\sup_{s\in [\tau,1]}\|v_s\|_{\CC^\gamma}
\leq (2+\|\bPsi\|_\mfB + \tau^{-1})^m
\leq (2+\|\bPsi\|_\mfB)^q\;.
\end{equ}
Therefore $\E\sup_{s\in [0,1]}\|v_s\|_{\CC^\gamma}^2 \leq (2+\fancynorm{u_0}_q)^q$.
By continuity of multiplication \eqref{eq:cont_mult}
we thus have
\begin{equ}
\E \sup_{s\in [0,t]} \|\Wick{P(u_s)} \|_{\CC^{\alpha p}(\T^d)}^2 \leq  (2+\fancynorm{u_0}_q)^q
\end{equ}
(which is analogous to the bound \eqref{eq:Pu_s_moment}).
The rest of the proof follows in the same way as the proof of Lemma \ref{lem:generator_classical}.
\end{proof}

\subsection{Non-Gaussianity of the invariant measure}
\label{sec:non_gauss}

Suppose in the remainder of this section that we are in the setting of Definition \ref{def:Markov}.
Let $\mu$ be a measure on $\CC^\alpha(\T^d)$ which is invariant for the Markov process $u_t$ from Definition \ref{def:Markov}.
By Lemma \ref{lem:u_in_mfC}, $\mu(\mfC) = 1$.

\begin{assumption}\label{ass:moments}
$\E_{u_0\sim \mu} \fancynorm{u_0}_q^q < \infty$ for all $q\in [1,\infty)$.
\end{assumption}
We assume for the rest of this section that Assumptions \ref{ass:covar_bound}, \ref{ass:P_DPD}, and \ref{ass:moments} hold.

\begin{remark}\label{rem:verify_ass}
To verify Assumption \ref{ass:moments},
it would suffice, by Lemma \ref{lem:Psi_plus_v}, that $P$ is linear so that $\mu$ is Gaussian, or that an a priori estimate holds of the form
\begin{equ}
\|v_1\|_{\CC^\gamma(\T^d)} \lesssim (2+\|\bPsi\|_{\mfB})^q
\end{equ}
\textit{uniformly} over initial conditions $v_0$
(this is a strengthening of Assumption \ref{ass:P_DPD}).
This form of `coming down from infinity' is known
for the $\Phi^4_2$ model \cite{Mourrat_Weber_17_Phi42, Tsatsoulis_Weber_18_SG}
and for more singular equations \cite{MoinatWeber20,CMW23}.
Since this moment estimate is a separate task, we prefer to leave it as an assumption.
\end{remark}

\begin{lemma}\label{lem:gen_poly}
% Suppose that Assumption \ref{ass:moments} holds.
Let $\phi \in \CC^\infty(\T^d)$ .
Then, for every integer $k \geq 1$,
\begin{equs}
\E_{u_0\sim\mu}
\Big[ k\scal{ u_0, \phi }^{k-1}(\scal{ u_0, (\Delta-1)\phi }
&+ \scal{ \Wick{P(u_0)} ,\phi})
\\
&+ \frac{k(k-1)}{2} \scal{ u_0, \phi }^{k-2} \scal{\Cov_\xi *\phi, \phi} \Big] = 0\;,
\end{equs}
where the final term in the expectation is zero if $k=1$.
\end{lemma}

\begin{proof}
This is identical to the proof of Lemma \ref{lem:gen_poly_classical}
upon substituting the use of Lemma \ref{lem:generator_classical} by Lemma \ref{lem:generator}.
\end{proof}

\begin{theorem}\label{thm:non_G}
Suppose the degree $p$ of $P$ is odd and $p\geq 3$.
%, and that additionally Assumption \ref{ass:moments} holds.
Then $\mu$ is not Gaussian.
\end{theorem}

\begin{proof}
Following the proof of Theorem \ref{thm:non_G_classical}, and using the shorthand $X = \scal{u_0,\varphi}$, $Y = \scal{\Wick{P(u_0)},\varphi}$, and $Z = \scal{u_0,(\Delta-1)\varphi}$ for $\phi\in\CC^\infty(\T^d)$,
Lemma \ref{lem:generator} implies the recursion $
\E{X^{k}(Y-Z)} = \frac{k}{2}\E{X^{k-1}}$
for all $k\geq 0$, where the right-hand side is understood as zero for $k=0$.

Suppose now that $\mu$ is Gaussian.
Similar to the the proof of Lemma \ref{lem:u_in_mfC},
we write $u_0\eqlaw u_1 = g + \Phi$ where $\Phi \sim \nu$ and $g\in \CC^\gamma(\T^d)$.
Let us write $K(x,y) = \E[\Phi(x)\Phi(y)]$.
By Assumption \ref{ass:covar_bound} and Young's convolution inequality, for $\frac{1}{r}+1 = \frac1p + \frac1q$,
\begin{equ}
\Big\|\int K(\cdot,y)\phi(y)\mrd y \Big\|_{L^r} \leq
\Big\|\int|\cdot-y|^\rho |\phi(y)| \mrd y
\Big\|_{L^r} \lesssim \|\phi\|_{L^q}
\end{equ}
where we used that $\rho p > -d$ so $\||\cdot|^\rho\|_{L^p}<\infty$.
Hence the convolution operator $K\colon L^q\to L^r$ is bounded.
Taking $r=2p$, so that $\frac1{2p} + \frac1q = 1$,
we obtain
\begin{equ}
\E\scal{\Phi,\phi}^2 = \scal{K\phi,\phi} \lesssim \|\phi\|_{L^q}^2\;.
\end{equ}
Furthermore, by Assumption \ref{ass:P_DPD}, $\E \|g\|_{\CC^\gamma}^q < \infty$ for all $q>0$.
Therefore
\begin{equ}
|\E\scal{\Phi,\phi}\scal{g,\psi}|
\leq \E[\scal{\Phi,\phi}^2]^{1/2}\E[\scal{g,\psi}^2]^{1/2}
\lesssim \|\phi\|_{L^q}\|\psi\|_{L^1}\;,
\end{equ}
and $\E\scal{g,\phi}^2 \lesssim \|\phi\|_{L^1}^2$.
It readily follows that
\begin{equ}
|\E\scal{u_0,\phi}\scal{u_0,\psi}| \lesssim \|\phi\|_{L^q}\|\psi\|_{L^q}\;.
\end{equ}
Hence, by duality, the covariance $C(x,y)= \E[u_0(x)u_0(y)]$ defines via convolution a bounded operator $C\colon L^q \to L^{2p}$.

Using the notation $Q_k$ and $\hat X$ from Theorem \ref{thm:non_G_classical}, it follows from an identical argument that
$\E Q_k(\hat X) Y = 0$ for all $k\geq 2$.
Recall also that $Q_k(\hat X)$ has the Wiener-It\^o decomposition $
Q_k (\hat X) = I_k (\phi^{\otimes k})$.

Using that $Y=\int_{\T^d} \Wick{P(u_0)}(x) \phi(x)\mrd x$,
it follows from the Wiener--It\^o decomposition of $\Wick{P(u_0)}$ at the top chaos\footnote{Note that $\Wick{P(u_0)}$ is generally in the \emph{inhomogeneous} $p$-th Wiener chaos of $\mu$ since the Wick ordering of $\Wick{P(u_0)}$ is with respect to $\nu$.} that
\begin{equ}
0 = \E Q_p(\hat X) Y
\propto \int_{\T^d}  (C\phi)(x)^{p}
% R\Big(\int_{\T} (C\phi)^2,\ldots, \int_{\T} (C\phi)^{2m} \Big)
\phi(x) \mrd x\;,
\end{equ}
where $p\geq 2$ is the degree of $P$.
By considering a smooth approximation of a non-zero real eigenfunction of $C$
 with eigenvalue $\lambda > 0$
 and using that $C\colon L^2\to L^{2p}$ is bounded, we obtain a contradiction in the same way as the end of the proof of Theorem \ref{thm:non_G_classical}.
\end{proof}

\begin{remark}\label{rem:Phi32}
Similar to Remark \ref{rem:Phi31}, our method of proof can handle more general polynomial non-linearities of the type~\eqref{eq:general_P}.
For example, it shows the invariant measure of
\begin{equ}
(\partial_t +1 -\Delta) u = u^{:2:} - \Big(\int_{\T^2} {u^{:2:}}\Big) u + \xi\;, \qquad u_0\in \CC^{\alpha}(\T^2)\;,
\end{equ}
which is the $\Phi^3_2$ measure
\begin{equ}
\mu(\mrd u) = Z^{-1} \exp\Big( -\frac13\int_{\T^2} {u^{:3:}} - \frac14\Big(\int_{\T^2} {u^{:2:}}\Big)^2 \Big) \nu(\mrd u)\;,
\end{equ}
is non-Gaussian.
\end{remark}

Another SPDE that is covered by the Da Prato--Debussche argument is the dynamical Sine-Gordon model in two spatial dimensions, given by
\begin{equ}\label{eq:SPDE_SG}
(\partial_t +1 -\Delta ) u = \Wick{\sin(\beta u)} + \sqrt{2} \xi\;,
\end{equ}
where $\beta \in \R$ with $\beta^{2} < 4 \pi$ and $\xi$ is a space-time white noise on $\R \times \T^{2}$.
Here the Wick trigonometrical functions are defined by setting 
\begin{equs}[e:wick_trig]
\Wick{\sin(\beta \bullet )} &= \lim_{\eps \downarrow 0} e^{\beta^{2} \E\Psi_{\eps}^{2}(0)} \sin(\beta \bullet )\;,\\
\Wick{\cos(\beta \bullet )} &= \lim_{\eps \downarrow 0}  e^{\beta^{2} \E\Psi_{\eps}^{2}(0)} \cos(\beta \bullet )\;,
\end{equs}
where $\Psi$ is the stationary-in-time solution to the linear equation. 

Writing $u = \Psi + v$ as before, one can show (again, with the restriction $\beta^{2} < 4 \pi$) the local well-posedness of the equation
\begin{equ}\label{eq:SG_remainder_eq}
(\partial_t - \Delta + 1) v = \sin( \beta v) \Wick{\cos( \beta \Psi)} + \cos( \beta v) \Wick{\sin(\beta \Psi)} 
\end{equ}

Local well-posedness for the dynamic \eqref{eq:SPDE_SG} in a probabilistically strong sense was first studied in \cite{HaoSG}.
Since \eqref{eq:SPDE_SG} is the Langevin dynamic for the Sine-Gordon Euclidean Quantum Field Theory, it would be interesting to see 
if the arguments of this article could be applied to show that any invariant measure of \eqref{eq:SPDE_SG} is non-Gaussian. 
This seems plausible, but there are some differences in this model that would require modifications to the argument. 

One difference is that $\Wick{\sin(\beta \Psi)}$ does not have a ``top level Wiener chaos'' like the polynomial interactions do, but this by itself is not an obstruction since $\Wick{\sin(\beta \Psi)}$ is in $L^{2}(\Omega,\nu)$ with infinitely many chaos contributions. 
The tricky step seems to be ruling out that, if $\mu$ is an invariant measure of \eqref{eq:SPDE_SG}, then $\mu$ is not a Gaussian with a covariance that has a stronger blow-up than the covariance of the stationary distribution of $\Psi$.  
In the Wick power case, the existence of $\Wick{u^{2}}$
implies that the blow-ups of the covariance of $\mu$ and the covariance of $\Psi$ have to be comparable. However, the same argument does not work for $\Wick{\sin(\beta \Psi)}$ since the renormalization in \eqref{e:wick_trig} is blowing up something going to $0$ rather than cancelling out a divergence. 

\appendix

\section{Da Prato--Debussche regime vs. absolute continuity}
\label{app:abs_cont}

In this appendix, we show that the DPD trick, as described n Section \ref{subsec:DPD_trick}, works in situations
where the invariant measure of $u$ is singular with respect to $\nu$, the Gaussian measure from Definition \ref{def:mfA}, which is invariant for the equation $(\partial_t +1-\Delta)\Psi = \xi$.

For simplicity, suppose $\Cov_\xi = (1-\Delta)^{\beta}$ where $\beta\leq 0$ and $\rho\eqdef 2-d-2\beta\leq 0$.
(In particular, we suppose $d\geq 2$.) 
Then
\begin{equ}
|(1-\Delta)^{-1}*\Cov_\xi(x)| = |(1-\Delta)^{-1+\beta}(x)| \lesssim
\begin{cases}
1+|\log |x| |  &\text{if } \rho = 0
\\
|x|^\rho  &\text{if } \rho<0
\end{cases}
\end{equ}
(see, e.g. the proof of \cite[Prop.~A.1]{CM24}).
To apply the results of Section \ref{subsec:DPD_trick},
due to Assumption \ref{ass:covar_bound}, we require $\rho>2\alpha$ and $\rho > -d/p$ for some $\alpha <0$.
Due to \eqref{eq:beta_assump}, we furthermore require $\alpha > 2/(1-2p)$.
To ensure the existence of such $\alpha$, we thus need
\begin{equ}[eq:DPD_rho]
\rho > \max \{4/(1-2p) , -d/p\}\;.
\end{equ}

On the other hand, a recent result of \cite{Hairer_24_singular} is that $\Law(u_t)$ for $t>0$ (and thus any invariant measure of $u$) is singular with respect to $\nu$.
With the notation of \cite[Sec.~3]{Hairer_24_singular} on the left-hand side and ours on the right-hand side, we have $m=\beta$, $\sigma=2$, $k=p$, $n_i=0$, and $A = \rho p/2$.
Therefore \cite[Assumption 3.2]{Hairer_24_singular} corresponds to $\rho<0$.
Furthermore the first part of \cite[Assumption 3.3]{Hairer_24_singular} is equivalent to $\rho p + d + 2 > 0$, which is weaker than our assumption $\rho p + d>0$.\footnote{This discrepancy is due to our necessity of having the $p$-th Wick power defined as a function of time and not just as a space-time distribution.}
The second part of \cite[Assumption 3.3]{Hairer_24_singular} is equivalent to $\rho p + d + 2 > -2\beta = \rho - 2 +d$, i.e. $\rho > -4/(p-1)$, which is again weaker than $\rho > 4/(1-2p)$.\footnote{This discrepancy is due to us considering the DPD regime, while the whole subcritical regime is considered in \cite{Hairer_24_singular}.}
It is also simple to verify that \eqref{eq:DPD_rho} implies the final \cite[Assumption 3.5]{Hairer_24_singular}.

It thus follows from \cite[Thm~3.9]{Hairer_24_singular} that $\Law(u_t)$ is singular with respect to $\nu$ whenever \eqref{eq:DPD_rho} holds, $\rho<0$, and
\begin{equ}[eq:abs_cont1]
\rho p  + 2 \leq -2\beta = \rho +d - 2\;.
\end{equ}
For $p > 1$, \eqref{eq:abs_cont1} is equivalent to
\begin{equ}[eq:abs_cont2]
\rho  \leq \frac{d-4}{p -1}
\end{equ}
while for $p=1$, \eqref{eq:abs_cont1} is equivalent simply to $d\geq 4$.

% For absolute continuity, recall that the Cameron--Martin space of $\nu$ is $H^{1-\beta}$.
% The following rule of thumb was recently shown in \cite{Hairer_24_singular}:
% if $u_t - \Psi_t$ is not in the Cameron--Martin space of $\nu$, then the law of $u_t$ and $\nu$ are mutually singular (in particular the invariant measure of $u$ and $\nu$ are mutually singular).

We now claim that, for $p\geq 3$ and $d>2$, there exists $\rho< 0$ satisfying \eqref{eq:DPD_rho} and \eqref{eq:abs_cont2}.
(For $d=2$, the only possibility is $\rho=\beta=0$, in which case the result of \cite{Hairer_24_singular} does not apply directly.)
Indeed,
\begin{itemize}
\item
for $d\geq 4$, clearly all $\rho\leq 0$ satisfy \eqref{eq:abs_cont2}, so the claim follows from the fact that the right-hand side of \eqref{eq:DPD_rho} is strictly negative,

\item for $d=3$, 
\eqref{eq:DPD_rho} is equivalent to $\rho > 4/(1-2p)$ while \eqref{eq:abs_cont2} is equivalent to $\rho \leq -1/(p-1)$,
so the claim follows from the fact that $4/(1-2p) < -1/(p-1)$ for $p>\frac32$.
\end{itemize}

\begin{example}\label{ex:Phi4}
Take $p=3$ and $d=3$. Then, by \eqref{eq:abs_cont2}, we have singularity of $\Law(u_t)$ with respect to $\nu$ for $\rho \leq -1/2$, i.e. $\alpha < -1/4$, which
corresponds to the $\Phi^4_{\delta}$ model with $\delta\geq \frac52$ in
the setting of \cite[Sec.~2.8.2]{BCCH17} (with $\T^3$ taken therein instead of $\T^4$).
On the other hand, the DPD regime, by \eqref{eq:DPD_rho}, requires only $\rho > -4/5$, i.e. $\alpha > -2/5$,
which corresponds to the $\Phi^4_{\delta}$ model with $\delta<\frac{14}{5}$.
\end{example}

\noindent
{\textbf{Acknowledgements}.}
IC gratefully acknowledges support from the DFG CRC/TRR 388 ``Rough
Analysis, Stochastic Dynamics and Related Fields'' through a Mercator Fellowship held at TU Berlin.
AC gratefully acknowledges partial support by the EPSRC through EP/S023925/1 through the “Mathematics
of Random Systems” CDT EP/S023925/1.
AC and IC thank the referee for helpful suggestions. 

\endappendix
\bibliographystyle{./Martin}
\bibliography{./refs}

\def\cprime{$'$} \def\polhk#1{\setbox0=\hbox{#1}{\ooalign{\hidewidth
  \lower1.5ex\hbox{`}\hidewidth\crcr\unhbox0}}}
\begin{thebibliography}{BCCH21}
\expandafter\ifx\csname url\endcsname\relax
  \def\url#1{\texttt{#1}}\fi
\expandafter\ifx\csname urlprefix\endcsname\relax\def\urlprefix{URL }\fi
\expandafter\ifx\csname href\endcsname\relax
  \def\href#1#2{#2}\fi
\expandafter\ifx\csname burlalt\endcsname\relax
  \def\burlalt#1#2{\href{#2}{\texttt{#1}}}\fi

\bibitem[AK20]{AK20}
\textsc{S.~Albeverio} and \textsc{S.~Kusuoka}.
\newblock The invariant measure and the flow associated to the
  {$\Phi^4_3$}-quantum field model.
\newblock \emph{Ann. Sc. Norm. Super. Pisa Cl. Sci. (5)} \textbf{20}, no.~4,
  (2020), 1359--1427.
\newblock
  \burlalt{doi:10.2422/2036-2145.201809_008}{http://dx.doi.org/10.2422/2036-2145.201809_008}.

\bibitem[AK25]{Albeverio_Kusuoka_21_rotation}
\textsc{S.~Albeverio} and \textsc{S.~Kusuoka}.
\newblock Construction of a {N}on-{G}aussian and {R}otation-{I}nvariant
  {$\Phi^4$}-{M}easure and {A}ssociated {F}low on {$\Bbb R^3$} {T}hrough
  {S}tochastic {Q}uantization.
\newblock \emph{Mem. Amer. Math. Soc.} \textbf{308}, no. 1558, (2025), v+114.
\newblock \burlalt{doi:10.1090/memo/1558}{http://dx.doi.org/10.1090/memo/1558}.

\bibitem[BC24a]{BC24SG}
\textsc{B.~{Bringmann}} and \textsc{S.~{Cao}}.
\newblock {Global well-posedness of the dynamical sine-Gordon model up to
  $6\pi$}.
\newblock \emph{arXiv e-prints} (2024).
\newblock \burlalt{arXiv:2410.15493}{http://arxiv.org/abs/2410.15493}.

\bibitem[BC24b]{BringmannCaoHiggs}
\textsc{B.~{Bringmann}} and \textsc{S.~{Cao}}.
\newblock {Global well-posedness of the stochastic Abelian-Higgs equations in
  two dimensions}.
\newblock \emph{arXiv e-prints} (2024).
\newblock \burlalt{arXiv:2403.16878}{http://arxiv.org/abs/2403.16878}.

\bibitem[BCCH21]{BCCH17}
\textsc{Y.~Bruned}, \textsc{A.~Chandra}, \textsc{I.~Chevyrev}, and
  \textsc{M.~Hairer}.
\newblock Renormalising {SPDE}s in regularity structures.
\newblock \emph{J. Eur. Math. Soc. (JEMS)} \textbf{23}, no.~3, (2021),
  869--947.
\newblock \burlalt{doi:10.4171/jems/1025}{http://dx.doi.org/10.4171/jems/1025}.

\bibitem[BCF18]{BCF18_renorm_SDEs}
\textsc{Y.~Bruned}, \textsc{I.~Chevyrev}, and \textsc{P.~K. Friz}.
\newblock Examples of renormalized {SDE}s.
\newblock In \emph{Stochastic partial differential equations and related
  fields}, vol. 229 of \emph{Springer Proc. Math. Stat.},  303--317. Springer,
  Cham, 2018.
\newblock
  \burlalt{doi:10.1007/978-3-319-74929-7\_19}{http://dx.doi.org/10.1007/978-3-319-74929-7\_19}.

\bibitem[BCFP19]{BCFP19}
\textsc{Y.~Bruned}, \textsc{I.~Chevyrev}, \textsc{P.~K. Friz}, and
  \textsc{R.~Prei\ss}.
\newblock A rough path perspective on renormalization.
\newblock \emph{J. Funct. Anal.} \textbf{277}, no.~11, (2019), 108283, 60.
\newblock
  \burlalt{doi:10.1016/j.jfa.2019.108283}{http://dx.doi.org/10.1016/j.jfa.2019.108283}.

\bibitem[Ber22]{Berglund_22_SPDEs}
\textsc{N.~Berglund}.
\newblock \emph{An introduction to singular stochastics {PDE}s---{A}llen-{C}ahn
  equations, metastability, and regularity structures}.
\newblock EMS Series of Lectures in Mathematics. EMS Press, Berlin, [2022]
  \copyright 2022,  x+220.
\newblock \burlalt{doi:10.4171/ELM/34}{http://dx.doi.org/10.4171/ELM/34}.

\bibitem[BFS83]{BFS83_Phi4}
\textsc{D.~C. Brydges}, \textsc{J.~Fr\"{o}hlich}, and \textsc{A.~D. Sokal}.
\newblock A new proof of the existence and nontriviality of the continuum
  {$\varphi \sp{4}\sb{2}$} and {$\varphi \sp{4}\sb{3}$} quantum field theories.
\newblock \emph{Comm. Math. Phys.} \textbf{91}, no.~2, (1983), 141--186.
\newblock
  \burlalt{doi:10.1007/BF01211157}{http://dx.doi.org/10.1007/BF01211157}.

\bibitem[BG20]{BG20}
\textsc{N.~Barashkov} and \textsc{M.~Gubinelli}.
\newblock A variational method for {$\Phi^4_3$}.
\newblock \emph{Duke Math. J.} \textbf{169}, no.~17, (2020), 3339--3415.
\newblock
  \burlalt{doi:10.1215/00127094-2020-0029}{http://dx.doi.org/10.1215/00127094-2020-0029}.

\bibitem[BHST87]{BHST87I}
\textsc{Z.~Bern}, \textsc{M.~B. Halpern}, \textsc{L.~Sadun}, and
  \textsc{C.~Taubes}.
\newblock Continuum regularization of quantum field theory. {I}. {S}calar
  prototype.
\newblock \emph{Nuclear Phys. B} \textbf{284}, no.~1, (1987), 1--34.
\newblock
  \burlalt{doi:10.1016/0550-3213(87)90025-3}{http://dx.doi.org/10.1016/0550-3213(87)90025-3}.

\bibitem[BHZ19]{BHZ19}
\textsc{Y.~Bruned}, \textsc{M.~Hairer}, and \textsc{L.~Zambotti}.
\newblock Algebraic renormalisation of regularity structures.
\newblock \emph{Invent. Math.} \textbf{215}, no.~3, (2019), 1039--1156.
\newblock
  \burlalt{doi:10.1007/s00222-018-0841-x}{http://dx.doi.org/10.1007/s00222-018-0841-x}.

\bibitem[CC18]{CC18_Phi43}
\textsc{R.~Catellier} and \textsc{K.~Chouk}.
\newblock Paracontrolled distributions and the 3-dimensional stochastic
  quantization equation.
\newblock \emph{Ann. Probab.} \textbf{46}, no.~5, (2018), 2621--2679.
\newblock
  \burlalt{doi:10.1214/17-AOP1235}{http://dx.doi.org/10.1214/17-AOP1235}.

\bibitem[CCHS22]{CCHS_2D}
\textsc{A.~{Chandra}}, \textsc{I.~{Chevyrev}}, \textsc{M.~{Hairer}}, and
  \textsc{H.~{Shen}}.
\newblock {Langevin dynamic for the 2D Yang-Mills measure}.
\newblock \emph{Publ. math. IHÉS} (2022).
\newblock \burlalt{arXiv:2006.04987}{http://arxiv.org/abs/2006.04987}.
\newblock
  \burlalt{doi:10.1007/s10240-022-00132-0}{http://dx.doi.org/10.1007/s10240-022-00132-0}.

\bibitem[CCHS24]{CCHS_3D}
\textsc{A.~Chandra}, \textsc{I.~Chevyrev}, \textsc{M.~Hairer}, and
  \textsc{H.~Shen}.
\newblock Stochastic quantisation of {Y}ang-{M}ills-{H}iggs in 3{D}.
\newblock \emph{Invent. Math.} \textbf{237}, no.~2, (2024), 541--696.
\newblock
  \burlalt{doi:10.1007/s00222-024-01264-2}{http://dx.doi.org/10.1007/s00222-024-01264-2}.

\bibitem[CF24]{CF_24_tensor}
\textsc{A.~Chandra} and \textsc{L.~Ferdinand}.
\newblock A stochastic analysis approach to tensor field theories.
\newblock \emph{arXiv e-prints} (2024).
\newblock To appear in \textit{Ann. Inst. Henri Poincaré Probab. Stat.}
\newblock \burlalt{arXiv:2306.05305}{http://arxiv.org/abs/2306.05305}.

\bibitem[CFW24]{CFW24}
\textsc{A.~{Chandra}}, \textsc{G.~d.~L. {Feltes}}, and \textsc{H.~{Weber}}.
\newblock {A priori bounds for 2-d generalised Parabolic Anderson Model}.
\newblock \emph{arXiv e-prints} (2024).
\newblock \burlalt{arXiv:2402.05544}{http://arxiv.org/abs/2402.05544}.

\bibitem[CH16]{CH16}
\textsc{A.~{Chandra}} and \textsc{M.~{Hairer}}.
\newblock {An analytic BPHZ theorem for regularity structures}.
\newblock \emph{ArXiv e-prints} (2016).
\newblock \burlalt{arXiv:1612.08138}{http://arxiv.org/abs/1612.08138}.

\bibitem[{Che}22]{Chevyrev22_Hopf_lectures}
\textsc{I.~{Chevyrev}}.
\newblock {Hopf and pre-Lie algebras in regularity structures}.
\newblock \emph{arXiv e-prints} (2022).
\newblock \burlalt{arXiv:2206.14557}{http://arxiv.org/abs/2206.14557}.

\bibitem[Che24]{Chevyrev22_norm_inf}
\textsc{I.~Chevyrev}.
\newblock Norm inflation for a non-linear heat equation with gaussian initial
  conditions.
\newblock \emph{Stoch. Partial Differ. Equ. Anal. Comput.} \textbf{12}, no.~3,
  (2024), 1745--1768.
\newblock
  \burlalt{doi:10.1007/s40072-023-00317-6}{http://dx.doi.org/10.1007/s40072-023-00317-6}.

\bibitem[CHM23]{CHM23}
\textsc{I.~Chevyrev}, \textsc{B.~Hambly}, and \textsc{A.~Mayorcas}.
\newblock A stochastic model of chemorepulsion with additive noise and
  nonlinear sensitivity.
\newblock \emph{Stoch. Partial Differ. Equ. Anal. Comput.} \textbf{11}, no.~2,
  (2023), 730--772.
\newblock
  \burlalt{doi:10.1007/s40072-022-00244-y}{http://dx.doi.org/10.1007/s40072-022-00244-y}.

\bibitem[CHS18]{Chandra_Hairer_Shen_18_SG}
\textsc{A.~Chandra}, \textsc{M.~Hairer}, and \textsc{H.~Shen}.
\newblock The dynamical sine-gordon model in the full subcritical regime.
\newblock \emph{arXiv e-prints} (2018).
\newblock \burlalt{arXiv:1808.02594}{http://arxiv.org/abs/1808.02594}.

\bibitem[CM24]{CM24}
\textsc{I.~{Chevyrev}} and \textsc{H.~{Mirsajjadi}}.
\newblock {Local well-posedness of subcritical non-linear heat equations with
  Gaussian initial data}.
\newblock \emph{arXiv e-prints} (2024).
\newblock \burlalt{arXiv:2410.11638}{http://arxiv.org/abs/2410.11638}.

\bibitem[CMW23]{CMW23}
\textsc{A.~Chandra}, \textsc{A.~Moinat}, and \textsc{H.~Weber}.
\newblock A {P}riori {B}ounds for the $\phi^4$ {E}quation in the {F}ull
  {S}ub-critical {R}egime.
\newblock \emph{Archive for {R}ational {M}echanics and {A}nalysis}
  \textbf{247}, no.~3, (2023), 48.
\newblock
  \burlalt{doi:10.1007/s00205-023-01876-7}{http://dx.doi.org/10.1007/s00205-023-01876-7}.

\bibitem[COW22]{COW_22_norm_inf}
\textsc{I.~{Chevyrev}}, \textsc{T.~{Oh}}, and \textsc{Y.~{Wang}}.
\newblock {Norm inflation for the cubic nonlinear heat equation above the
  scaling critical regularity}.
\newblock \emph{arXiv e-prints} (2022).
\newblock To appear in \textit{Funkcialaj Ekvacioj}.
\newblock \burlalt{arXiv:2205.14488}{http://arxiv.org/abs/2205.14488}.

\bibitem[CS23]{CS23_invariant}
\textsc{I.~Chevyrev} and \textsc{H.~Shen}.
\newblock Invariant measure and universality of the 2d yang-mills langevin
  dynamic, 2023.
\newblock \burlalt{arXiv:2302.12160}{http://arxiv.org/abs/2302.12160}.
\newblock \urlprefix\url{https://arxiv.org/abs/2302.12160}.

\bibitem[CW17]{Chandra_Weber_17_SPDEs}
\textsc{A.~Chandra} and \textsc{H.~Weber}.
\newblock Stochastic {PDE}s, regularity structures, and interacting particle
  systems.
\newblock \emph{Ann. Fac. Sci. Toulouse Math. (6)} \textbf{26}, no.~4, (2017),
  847--909.
\newblock \burlalt{doi:10.5802/afst.1555}{http://dx.doi.org/10.5802/afst.1555}.

\bibitem[CZ20]{Caravenna_Zambotti_20_Reconstruct}
\textsc{F.~Caravenna} and \textsc{L.~Zambotti}.
\newblock Hairer's reconstruction theorem without regularity structures.
\newblock \emph{EMS Surv. Math. Sci.} \textbf{7}, no.~2, (2020), 207--251.
\newblock \burlalt{doi:10.4171/emss/39}{http://dx.doi.org/10.4171/emss/39}.

\bibitem[DGR24]{DGR_24_fractional}
\textsc{P.~Duch}, \textsc{M.~Gubinelli}, and \textsc{P.~Rinaldi}.
\newblock Parabolic stochastic quantisation of the fractional $\phi^4_3$ model
  in the full subcritical regime.
\newblock \emph{arXiv e-prints} (2024).
\newblock \burlalt{arXiv:2303.18112}{http://arxiv.org/abs/2303.18112}.

\bibitem[Dim74]{Dimock74}
\textsc{J.~Dimock}.
\newblock Asymptotic perturbation expansion in the {$P(\phi)\sb{2}$} quantum
  field theory.
\newblock \emph{Comm. Math. Phys.} \textbf{35}, (1974), 347--356.
\newblock
  \burlalt{doi:10.1007/BF01646354}{http://dx.doi.org/10.1007/BF01646354}.

\bibitem[DPD02]{DPD02_SNS}
\textsc{G.~Da~Prato} and \textsc{A.~Debussche}.
\newblock Two-dimensional {N}avier-{S}tokes equations driven by a space-time
  white noise.
\newblock \emph{J. Funct. Anal.} \textbf{196}, no.~1, (2002), 180--210.
\newblock
  \burlalt{doi:10.1006/jfan.2002.3919}{http://dx.doi.org/10.1006/jfan.2002.3919}.

\bibitem[DPD03]{DPD03_SQ}
\textsc{G.~Da~Prato} and \textsc{A.~Debussche}.
\newblock Strong solutions to the stochastic quantization equations.
\newblock \emph{Ann. Probab.} \textbf{31}, no.~4, (2003), 1900--1916.
\newblock
  \burlalt{doi:10.1214/aop/1068646370}{http://dx.doi.org/10.1214/aop/1068646370}.

\bibitem[DPT23]{Da_Prato_Tubaro_23_Wick}
\textsc{G.~Da~Prato} and \textsc{L.~Tubaro}.
\newblock Wick powers in stochastic {PDE}s: an introduction.
\newblock In \emph{Quantum and stochastic mathematical physics}, vol. 377 of
  \emph{Springer Proc. Math. Stat.},  1--15. Springer, Cham, [2023] \copyright
  2023.
\newblock
  \burlalt{doi:10.1007/978-3-031-14031-0\_1}{http://dx.doi.org/10.1007/978-3-031-14031-0\_1}.

\bibitem[Duc25]{Duch21}
\textsc{P.~Duch}.
\newblock Flow equation approach to singular stochastic {PDE}s.
\newblock \emph{Probab. Math. Phys.} \textbf{6}, no.~2, (2025), 327--437.
\newblock
  \burlalt{doi:10.2140/pmp.2025.6.327}{http://dx.doi.org/10.2140/pmp.2025.6.327}.

\bibitem[EW24]{EW_24_fractional}
\textsc{S.~Esquivel} and \textsc{H.~Weber}.
\newblock A priori bounds for the dynamic fractional $\phi^4$ model on
  $\mathbb{T}^3$ in the full subcritical regime.
\newblock \emph{arXiv e-prints} (2024).
\newblock \burlalt{arXiv:2411.16536}{http://arxiv.org/abs/2411.16536}.

\bibitem[GH21]{GH21}
\textsc{M.~Gubinelli} and \textsc{M.~Hofmanov\'{a}}.
\newblock A {PDE} construction of the {E}uclidean {$\phi_3^4$} quantum field
  theory.
\newblock \emph{Comm. Math. Phys.} \textbf{384}, no.~1, (2021), 1--75.
\newblock
  \burlalt{doi:10.1007/s00220-021-04022-0}{http://dx.doi.org/10.1007/s00220-021-04022-0}.

\bibitem[GHR24]{GHN_24_decay}
\textsc{M.~Gubinelli}, \textsc{M.~Hofmanov{\'a}}, and \textsc{N.~Rana}.
\newblock Decay of correlations in stochastic quantization: the exponential
  euclidean field in two dimensions.
\newblock \emph{Stochastics and Partial Differential Equations: Analysis and
  Computations} (2024).
\newblock
  \burlalt{doi:10.1007/s40072-024-00328-x}{http://dx.doi.org/10.1007/s40072-024-00328-x}.

\bibitem[GIP15]{GIP15}
\textsc{M.~Gubinelli}, \textsc{P.~Imkeller}, and \textsc{N.~Perkowski}.
\newblock Paracontrolled distributions and singular {PDE}s.
\newblock \emph{Forum Math. Pi} \textbf{3}, (2015), e6, 75.
\newblock \burlalt{arXiv:1210.2684v3}{http://arxiv.org/abs/1210.2684v3}.
\newblock
  \burlalt{doi:10.1017/fmp.2015.2}{http://dx.doi.org/10.1017/fmp.2015.2}.

\bibitem[GJ87]{GlimmJaffe}
\textsc{J.~Glimm} and \textsc{A.~Jaffe}.
\newblock \emph{Quantum physics}.
\newblock Springer-Verlag, New York, second ed., 1987,  xxii+535.
\newblock A functional integral point of view.
\newblock
  \burlalt{doi:10.1007/978-1-4612-4728-9}{http://dx.doi.org/10.1007/978-1-4612-4728-9}.

\bibitem[GM24]{GM24}
\textsc{M.~{Gubinelli}} and \textsc{S.-J. {Meyer}}.
\newblock {The FBSDE approach to sine-Gordon up to $6\pi$}.
\newblock \emph{arXiv e-prints} (2024).
\newblock \burlalt{arXiv:2401.13648}{http://arxiv.org/abs/2401.13648}.

\bibitem[GP20]{Gubinelli_Perkowski_20_generator}
\textsc{M.~Gubinelli} and \textsc{N.~Perkowski}.
\newblock The infinitesimal generator of the stochastic {B}urgers equation.
\newblock \emph{Probab. Theory Related Fields} \textbf{178}, no. 3-4, (2020),
  1067--1124.
\newblock
  \burlalt{doi:10.1007/s00440-020-00996-5}{http://dx.doi.org/10.1007/s00440-020-00996-5}.

\bibitem[Hai13]{KPZ}
\textsc{M.~Hairer}.
\newblock Solving the {KPZ} equation.
\newblock \emph{Ann. of Math. (2)} \textbf{178}, no.~2, (2013), 559--664.
\newblock \burlalt{arXiv:1109.6811}{http://arxiv.org/abs/1109.6811}.
\newblock
  \burlalt{doi:10.4007/annals.2013.178.2.4}{http://dx.doi.org/10.4007/annals.2013.178.2.4}.

\bibitem[Hai14a]{Hairer14_ICM}
\textsc{M.~Hairer}.
\newblock Singular stochastic {PDE}s.
\newblock In \emph{Proceedings of the {I}nternational {C}ongress of
  {M}athematicians---{S}eoul 2014. {V}ol. 1},  685--709. Kyung Moon Sa, Seoul,
  2014.

\bibitem[Hai14b]{Hairer14}
\textsc{M.~Hairer}.
\newblock A theory of regularity structures.
\newblock \emph{Invent. Math.} \textbf{198}, no.~2, (2014), 269--504.
\newblock \burlalt{arXiv:1303.5113}{http://arxiv.org/abs/1303.5113}.
\newblock
  \burlalt{doi:10.1007/s00222-014-0505-4}{http://dx.doi.org/10.1007/s00222-014-0505-4}.

\bibitem[HKN24]{Hairer_24_singular}
\textsc{M.~{Hairer}}, \textsc{S.~{Kusuoka}}, and \textsc{H.~{Nagoji}}.
\newblock {Singularity of solutions to singular SPDEs}.
\newblock \emph{arXiv e-prints} (2024).
\newblock \burlalt{arXiv:2409.10037}{http://arxiv.org/abs/2409.10037}.

\bibitem[HS16]{HaoSG}
\textsc{M.~Hairer} and \textsc{H.~Shen}.
\newblock The dynamical sine-{G}ordon model.
\newblock \emph{Comm. Math. Phys.} \textbf{341}, no.~3, (2016), 933--989.
\newblock \burlalt{arXiv:1409.5724}{http://arxiv.org/abs/1409.5724}.
\newblock
  \burlalt{doi:10.1007/s00220-015-2525-3}{http://dx.doi.org/10.1007/s00220-015-2525-3}.

\bibitem[HS22]{Hairer_Steele_22}
\textsc{M.~Hairer} and \textsc{R.~Steele}.
\newblock The {$\Phi_3^4$} measure has sub-{G}aussian tails.
\newblock \emph{J. Stat. Phys.} \textbf{186}, no.~3, (2022), Paper No. 38, 25.
\newblock
  \burlalt{doi:10.1007/s10955-021-02866-3}{http://dx.doi.org/10.1007/s10955-021-02866-3}.

\bibitem[Kup16]{Kupiainen2016}
\textsc{A.~Kupiainen}.
\newblock Renormalization group and stochastic {PDE}s.
\newblock \emph{Annales Henri Poincar{\'e}} \textbf{17}, no.~3, (2016),
  497--535.
\newblock \burlalt{arXiv:1410.3094}{http://arxiv.org/abs/1410.3094}.
\newblock
  \burlalt{doi:10.1007/s00023-015-0408-y}{http://dx.doi.org/10.1007/s00023-015-0408-y}.

\bibitem[LOTT24]{LOTT24}
\textsc{P.~Linares}, \textsc{F.~Otto}, \textsc{M.~Tempelmayr}, and
  \textsc{P.~Tsatsoulis}.
\newblock A diagram-free approach to the stochastic estimates in regularity
  structures.
\newblock \emph{Invent. Math.} \textbf{237}, no.~3, (2024), 1469--1565.
\newblock
  \burlalt{doi:10.1007/s00222-024-01275-z}{http://dx.doi.org/10.1007/s00222-024-01275-z}.

\bibitem[Lyo98]{Lyons}
\textsc{T.~J. Lyons}.
\newblock Differential equations driven by rough signals.
\newblock \emph{Rev. Mat. Iberoamericana} \textbf{14}, no.~2, (1998), 215--310.
\newblock \burlalt{doi:10.4171/RMI/240}{http://dx.doi.org/10.4171/RMI/240}.

\bibitem[MS77]{MS77}
\textsc{J.~Magnen} and \textsc{R.~S\'{e}n\'{e}or}.
\newblock Phase space cell expansion and {B}orel summability for the
  {E}uclidean {$\phi \sb{3}\sp{4}$} theory.
\newblock \emph{Comm. Math. Phys.} \textbf{56}, no.~3, (1977), 237--276.
\newblock
  \burlalt{doi:10.1007/BF01614211}{http://dx.doi.org/10.1007/BF01614211}.

\bibitem[MW17a]{MW17Phi43}
\textsc{J.-C. Mourrat} and \textsc{H.~Weber}.
\newblock The dynamic {$\Phi^4_3$} model comes down from infinity.
\newblock \emph{Comm. Math. Phys.} \textbf{356}, no.~3, (2017), 673--753.
\newblock \burlalt{arXiv:1601.01234}{http://arxiv.org/abs/1601.01234}.
\newblock
  \burlalt{doi:10.1007/s00220-017-2997-4}{http://dx.doi.org/10.1007/s00220-017-2997-4}.

\bibitem[MW17b]{Mourrat_Weber_17_Phi42}
\textsc{J.-C. Mourrat} and \textsc{H.~Weber}.
\newblock Global well-posedness of the dynamic {$\Phi^4$} model in the plane.
\newblock \emph{Ann. Probab.} \textbf{45}, no.~4, (2017), 2398--2476.
\newblock
  \burlalt{doi:10.1214/16-AOP1116}{http://dx.doi.org/10.1214/16-AOP1116}.

\bibitem[MW20a]{MoinatWeber20_RD}
\textsc{A.~Moinat} and \textsc{H.~Weber}.
\newblock Local bounds for stochastic reaction diffusion equations.
\newblock \emph{Electron. J. Probab.} \textbf{25}, (2020), Paper No. 17, 26.
\newblock \burlalt{doi:10.1214/19-ejp397}{http://dx.doi.org/10.1214/19-ejp397}.

\bibitem[MW20b]{MoinatWeber20}
\textsc{A.~Moinat} and \textsc{H.~Weber}.
\newblock Space-time localisation for the dynamic {$\Phi^4_3$} model.
\newblock \emph{Comm. Pure Appl. Math.} \textbf{73}, no.~12, (2020),
  2519--2555.
\newblock \burlalt{doi:10.1002/cpa.21925}{http://dx.doi.org/10.1002/cpa.21925}.

\bibitem[OW19]{OW}
\textsc{F.~Otto} and \textsc{H.~Weber}.
\newblock Quasilinear {SPDE}s via rough paths.
\newblock \emph{Arch. Ration. Mech. Anal.} \textbf{232}, no.~2, (2019),
  873--950.
\newblock
  \burlalt{doi:10.1007/s00205-018-01335-8}{http://dx.doi.org/10.1007/s00205-018-01335-8}.

\bibitem[PW81]{ParisiWu}
\textsc{G.~Parisi} and \textsc{Y.~S. Wu}.
\newblock Perturbation theory without gauge fixing.
\newblock \emph{Sci. Sinica} \textbf{24}, no.~4, (1981), 483--496.
\newblock
  \burlalt{doi:10.1360/ya1981-24-4-483}{http://dx.doi.org/10.1360/ya1981-24-4-483}.

\bibitem[She21]{Shen18}
\textsc{H.~Shen}.
\newblock Stochastic quantization of an {A}belian gauge theory.
\newblock \emph{Comm. Math. Phys.} \textbf{384}, no.~3, (2021), 1445--1512.
\newblock
  \burlalt{doi:10.1007/s00220-021-04114-x}{http://dx.doi.org/10.1007/s00220-021-04114-x}.

\bibitem[TW18]{Tsatsoulis_Weber_18_SG}
\textsc{P.~Tsatsoulis} and \textsc{H.~Weber}.
\newblock Spectral gap for the stochastic quantization equation on the
  2-dimensional torus.
\newblock \emph{Ann. Inst. Henri Poincar\'{e} Probab. Stat.} \textbf{54},
  no.~3, (2018), 1204--1249.
\newblock
  \burlalt{doi:10.1214/17-AIHP837}{http://dx.doi.org/10.1214/17-AIHP837}.

\end{thebibliography}

\end{document}